\documentclass[11pt,a4paper]{amsart}
\usepackage[headings]{fullpage}
\usepackage{amssymb,amscd,hyperref}
\usepackage{amsthm}
\usepackage[alphabetic, nobysame]{amsrefs}
\usepackage[all,cmtip]{xy}
\usepackage{color}
\usepackage{xcolor}
\usepackage{tikz-cd} 
\numberwithin{equation}{section}
\newtheorem{thm}{Theorem}[section]
\newtheorem{lem}[thm]{Lemma}
\newtheorem{prop}[thm]{Proposition}
\newtheorem{cor}[thm]{Corollary}

\theoremstyle{definition}
\newtheorem{defn}[thm]{Definition}
\newtheorem{ex}[thm]{Example}

\theoremstyle{remark}
\newtheorem{rem}[thm]{Remark}

\DeclareMathOperator{\colim}{colim}
\def\Z{\mathbb{Z}}
\def\Q{\mathbb{Q}}
\def\F{\mathbb{F}}

\def\GHH{\underline{\mathsf{HH}}}

\def\Tor{\mathsf{Tor}}

\def\mack{\mathsf{Mack}}

\def\tamb{\text{-}\mathsf{Tamb}}
\def\coind{\mathsf{Coind}}

\def\1mbf{\mathbf{1}}
\def\2mbf{\mathbf{2}}

\def\fix{\mathrm{fix}}

\newcommand{\und}[1]{{\underline{#1}}}
\def\norm{\mathsf{norm}}

\def\tr{\mathsf{tr}}
\def\res{\mathsf{res}}

\def\ra{\rightarrow}
\def\leq{\leqslant}
\def\geq{\geqslant}

\def\cL{\mathcal{L}}

\def\HH{\mathsf{HH}}
\def\HR{\mathsf{HR}}

\def\ie{\emph{i.e.}}

\def\id{\mathrm{id}}

\begin{document}
\title{Galois and separable extensions of Tambara functors}

\author{Birgit Richter}
\address{Fachbereich Mathematik der Universit\"at Hamburg,
  Bundesstra{\ss}e 55, 20146 Hamburg,  Germany}
\email{birgit.richter@uni-hamburg.de}

\date{\today}

\begin{abstract}
For a map of Tambara functors $\und{R} \ra \und{T}$ we define when
$\und{T}$ is separable over $\und{R}$ and if $\und{T}$ carries an
action by some finite group $H$ that fixes $\und{R}$ we also define
when $\und{R} \rightarrow \und{T}$ is an $H$-Galois extension of Tambara functors. We show
that flat and separable extensions are (formally) \'etale in the sense
of Hill and that Galois extensions are separable. We prove 
that the Burnside Tambara functor is separably closed and we express
Nullstellensatzian Tambara functors as filtered colimits of
Galois extensions. 
\end{abstract}
\maketitle  

\section{Introduction}

In classical algebra, the difference between ordinary commutative
rings and fields is huge. Their intrinsic algebraic properties are
very different and they differ 
drastically with respect to their category of modules: modules over a
ring can behave pretty erratic whereas linear algebra tells us how to
control modules (aka vector spaces) 
over fields. In analogy with the Galois theory of fields there is a
Galois theory for commutative rings, where one has to  
require more properties from an inclusion $R \subset T$ of commutative
rings so that it behaves similar to a $H$-Galois extension of fields
for a finite group $H$. 
In particular, on top of the usual condition that $R$ agrees with
$T^H$, the $H$-fixed points of $T$, one has to add a condition that
ensures that the extension is unramified by asking 
that $T \otimes_R T$ is isomorphic to $\prod_H T$ via the map that
sends $t_1 \otimes t_2$ to $(t_1\cdot h(t_2))_{h \in H}$. For instance
the extension $\Z \subset \Z[i]$ is 
\emph{not} $C_2$-Galois because the above map fails to be surjective,
but if one inverts the ramified prime $2$, then $\Z[\frac{1}{2}] \ra
\Z[\frac{1}{2},i]$ \emph{is} a $C_2$-Galois extension. 

For Tambara functors the distinction between ``fields'' and ordinary
Tambara functors is less prominent. Nakaoka defined field-like
$G$-Tambara functors \cite{nakaoka}. They can be characterized as
those $G$-Tambara functors  whose free level $\und{T}(G/e)$ 
has no non-trivial $G$-invariant ideals and that have injective
restriction maps associated to equivariant maps of finite transitive 
$G$-sets \cite[Theorem 4.32]{nakaoka}. However, this does \emph{not}
imply that all levels are fields or that modules over such a
field-like Tambara functor are necessarily free. 

Nonetheless, there are classification results for Tambara fields. For
instance Wisdom classifies field-like Tambara functors for cyclic
group of prime power order \cite{wisdom}. We also understand Tambara
functors that behave like algebraically closed fields thanks to work
of Schuchardt, Spitz, and Wisdom \cite{ssw}. But it was not clear how
to relate these notions to Galois theory. Can we interpret such an
algebraically closed object as sitting at the end of a chain of Galois
extensions? 

This paper aims to study Galois extensions of Tambara functors and to
relate them to \'etale and separable extensions.  

In  \cite{hill} Mike Hill defined formal \'etaleness in the context of Tambara
functors by introducing an equivariant version of the module of
K\"ahler differentials. These differentials are linear and satisfy a
Leibniz rule as in the non-equivariant case, but their equivariant
nature is reflected in the fact that they also have to obey a twisted
Leibniz rule that involves norms, transfers and restrictions. Examples
of formally \'etale extensions mostly originating in fixed point
Tambara functors for non-equivariant Galois
extensions were established in \cite{lrz}. Noah Wisdom extended this
range of examples drastically in \cite{wisdom-affine} and proved
several structural properties about (formally) \'etale maps \cite[\S
4]{wisdom-affine}.

In section \ref{sec:fixbox} we start with a small technical result
that expresses the  box product of two fixed point Tambara functors as
the fixed point Tambara functor of the 
tensor product if the group order is invertible. This generalizes 
\cite[Lemma 3.2]{lr} from $C_2$ to arbitrary finite groups.  
In \cite{mqs} K\"ahler differentials were related
to an equivariant version of Hochschild homology. We describe this
relation from a slightly different angle in Section \ref{sec:hh}.  

Non-equivariantly, extensions which are flat and separable are 
formally \'etale, and Galois extensions of commutative rings are
separable. Hence these types of algebras give natural examples for formally  
\'etale maps. In this paper we define separability (see Section
\ref{sec:etaleseparable}) and Galois extensions for Tambara
functors with analogous properties. As separability is defined via the
module category of a 
ring, there is a straightforward generalization in the equivariant
context. We use the concept of separability for Tambara functors to
prove that some equivariant Loday constructions in the sense of
\cite{lrz-loday} have vanishing homotopy groups in positive degrees
(see Section \ref{sec:appl-loday}).   

There is a well-developed notion of Galois extensions of commutative
rings \cite{auslandergoldman,chr,greither}. We define Galois
extensions of Tambara functors in a similar manner in Section 
\ref{sec:galois} and prove that Galois extensions of Tambara
functors are always separable. We also show that the Burnside Tambara
functor $\und{A}$ is separably closed, \ie, there are no connected $H$-Galois
extensions $\und{A} \ra \und{T}$ for non-trivial finite groups $H$ (see Theorem
\ref{thm:asepclosed}). This is analogous to the facts that the ring of
integers $\Z$ is separably closed in the category of commutative rings
and that the sphere spectrum is separably closed in the category of
commutative ring spectra \cite[Proposition 10.3.2, Theorem
10.3.3]{rognes}.

For some other properties that hold for Galois extensions
of commutative rings it is less clear that they carry over to the
setting of Tambara functors because some of the classical tools are
not (yet) available, such as local-to-global arguments. 

We study examples of such Galois extensions. In the
non-equivariant setting the above mentioned extension of commutative rings
$\Z[\frac{1}{2}] \ra \Z[\frac{1}{2},i]$ is a fundamental example. If
one wants to mimic its construction in the world of $C_2$-Tambara
functors by taking a suitable quotient of a free Tambara functors on a
generator at the trivial level, then this does not work, see Section
\ref{sec:nogauss}. We also show that one \emph{can} construct a
$C_2$-Green functor version of the above extension.   
  
Schuchardt, Spitz, and Wisdom \cite{ssw} investigate and classify Nullstellensatzian 
objects in the sense of Burklund, Schlank, and Yuan \cite{bsy} in the category of $G$-Tambara functors for any finite group $G$. They identify them as co-inductions $\coind_e^G(\mathbb{F})$ of
algebraically closed fields $\mathbb{F}$. We interpret these objects in the context of Galois theory of Tambara functors in Section \ref{sec:nullstellensatzian}.

\subsection*{Acknowledgements}

The author thanks Mike Hill, Niko Naumann, and Noah Wisdom for helpful
comments. I thank Ayelet Lindenstrauss and Foling Zou for a very
fruitful ongoing collaboration which led to my curiosity about
separability and Galois extensions in the equivariant context.  I would like to thank the Isaac Newton Institute for Mathematical
Sciences, Cambridge, for support and hospitality during the programme
Equivariant homotopy theory in context, where work on this paper was
undertaken. This work was supported by EPSRC grant EP/Z000580/1.

\section{Fixed point Tambara functors and box products} \label{sec:fixbox}

We start with a small, but very helpful technical result that will allow us to control box products of fixed point Tambara functors. 

Lewis noted several facts about fixed point and orbit Mackey functors
\cite[\S 1]{lewis}. We will use one of them to prove a generalization
of a formula for fixed point Tambara functors from \cite[Lemma
3.2]{lr}: There we showed that there is an isomorphism of $C_2$-Mackey
functors $\und{M}^\fix \Box \und{N}^\fix \cong
\und{(M \otimes N)}^\fix$, if $M$ and $N$ are abelian groups with a
$C_2$-action and if $2$ is invertible in one of them. We also proved
an upgrade to an isomorphism of $C_2$-Tambara functors if both $M$ and
$N$ are in fact commutative $C_2$-rings. It was also clear that the arugments could be extended to arbitrary cyclic groups of prime order $C_p$.
In the following
we extend the result to arbitrary finite groups as long as $|G|$ is
invertible in one of the factors. 

In the following, $G$ is an arbitrary but fixed finite group. Let $M$
be an abelian group with a $G$-action. We denote by $\und{L(M)}$ the
orbit Mackey functor. This has $\und{L(M)}(G/H) = M_H$, the
coinvariants of $M$ with respect to $H$. For $K < H$ the transfer map
$\tr_K^H \colon M_K \ra M_H$ is the canonical projection map and the
restriction $\res_K^H \colon M_H \ra M_K$ is given by sending an
equivalence class $[m] \in M_H$ to the class of $\sum_{\gamma \in H/K}
\gamma m \in M_K$. Here, the sum is taken over representatives for $K$
in $H$. 

Note that $\und{L}$ is a functor from the category of abelian groups with
$G$-action to the category of $G$-Mackey functors and is left adjoint to the evaluation functor that takes a $G$-Mackey functor and evaluates it at the free level $G/e$.

\begin{lem}
  Let $M$ and $N$ be abelian groups with a $G$-action. If $|G|$ is invertible in
  $M$ or $N$, then there is an isomorphism of $G$-Mackey functors
  \begin{equation} \label{eq:fix-iso}
    \und{M}^\fix \Box \und{N}^\fix \cong \und{(M \otimes N)}^\fix.\end{equation}

\end{lem}  
\begin{proof}
  Lewis states in \cite[Example 1.2 (b)]{lewis} that for all
  $G$-Mackey functors $\und{T}$ and all abelian groups with $G$-action
  $M$ 
  \[ \und{L(M)} \Box \und{T} \cong \und{L(M \otimes \und{T}(G/e))} \]
  where one considers the diagonal $G$-action on $M \otimes \und{T}(G/e)$. 

  The map $\und{L(M \otimes \und{T}(G/e))} \ra \und{L(M)} \Box
  \und{T}$ is determined by the adjoint map 
  \[ M \otimes \und{T}(G/e) \ra \und{L(M)} \Box \und{T}(G/e) = M \otimes \und{T}(G/e)\]
  which we take to be the identity map. For the converse map we use the description of maps out of a box product as componentwise maps
  \[\xymatrix{ \und{L(M)}(G/H) \otimes \und{T}(G/H) \ar@{=}[d] \ar[rr] 
& &      \und{L(M \otimes \und{T}(G/e))}(G/H)  \ar@{=}[d] &  \\
M_H \otimes \und{T}(G/H)  \ar@/_5ex/[rrr]_{\psi_H} &   &  (M \otimes \und{T}(G/e))_H \ar[r]^{\cong} &  M_H \otimes \und{T}(G/e)_H }\]
  that are compatible with the structure maps, see \cite[Remark 1.2.3 and p.~41]{mazur} for the explicit 
  list of conditions. We set $\psi_H$ to be $\id_{M_H} \otimes (\pi \circ \res_e^H)$ where $\pi \colon \und{T}(G/e) \ra \und{T}(G/e)_H$ is the projection map. These maps are inverse to each other. 

  If $|G|$ is invertible in $M$, then $\und{L(M)}$ is isomorphis to $\und{M}^\fix$ as $G$-Mackey functors: At level $G/H$ 
  \begin{equation} \label{eq:norm} M_H \ra M^H\end{equation} 
  is the map $[m] \mapsto \sum_{h \in H} hm$ and conversely, the map $M^H \ra M_H$ is given by sending an $H$-fixed point $m$ to $\frac{1}{|H|}[m]$.

  Taking these two isomorphisms together yields the desired isomorphism of
  $G$-Mackey functors: 
  \begin{align} \label{eq:isofix}
   \begin{split} 
\und{M}^\fix \Box \und{N}^\fix \cong & \und{L(M)} \Box \und{N}^\fix \\
    \cong & \und{L(M \otimes N)} \\
    \cong & \und{(M\otimes N)}^\fix.
    \end{split}
  \end{align}
\end{proof}  

\begin{cor} \label{cor:boxfix} 
If $\und{T}^\fix$ and $\und{R}^\fix$ are $G$-Tambara functors and if $|G|$ is invertible in $R$ or $T$, then the isomorphism \eqref{eq:fix-iso} is an isomorphism of $G$-Tambara functors. 
\end{cor}
\begin{proof}
Without loss of generality we  assume that $|G|$ is invertible in $T$. 
  The isomorphism \eqref{eq:isofix} uses the isomorphism $\und{L(T)} \Box \und{R}^ \fix \cong \underline{L(T \otimes R)}$ together with the map \eqref{eq:norm} and its inverse. Explicitly,  it sends $t \otimes r$ at level $G/K$ with $t \in T^K$ and $r \in R^K$ to
  \begin{align*} t \otimes r & \mapsto \frac{1}{|K|}[t] \otimes r \in T_K \otimes R^K \\
                             &    \mapsto  \frac{1}{|K|}[t] \otimes [\res_e^K(r)] \in T_K \otimes R_K \\
    &  \mapsto \sum_{k \in K} \frac{1}{|K|} kt \otimes k\res_e^K(r) \in (T \otimes R)^K, 
    \end{align*}
but as $t$ and $r$ are $K$-fixed points and as the restriction map in fixed point Tambara functors is the inclusion, this agrees with $t \otimes r$, so in total the map includes $T^K \otimes R^K$ into $(T\otimes R)^K$ and this is a ring map and is compatible with the norm. 
  \end{proof}

If the order of the group $G$ is neither invertible in $M$ nor in $N$, then the box product of two fixed point $G$-Tambara functors does not necessarily simplify to the fixed point functor of the tensor product: 
\begin{ex}
  Let $M = \Z[i] = N$. We claim that
  \[ \und{\Z[i]}^\fix \Box \und{\Z[i]}^\fix \not\cong \und{(\Z[i] \otimes \Z[i])}^\fix. \]

The two Tambara functors agree at the free level:
\[ \und{\Z[i]}^\fix \Box \und{\Z[i]}^\fix(C_2/e) = \Z[i] \otimes \Z[i]
= \und{(\Z[i] \otimes \Z[i])}^\fix(C_2/e). \]

However, at the trivial level we have
\[ \und{(\Z[i] \otimes \Z[i])}^\fix(C_2/C_2) = (\Z[i] \otimes
  \Z[i])^{C_2}\]
and the fixed points in the tensor product are the abelian group
spanned by $1 \otimes 1$ and $i \otimes i$.

For $\und{\Z[i]}^\fix \Box \und{\Z[i]}^\fix(C_2/C_2)$ we obtain
\[ \Big(\Z \otimes \Z \oplus (\Z[i] \otimes
  \Z[i])/W\Big)/\mathrm{FR}\]
where $W$ denotes the Weyl group action that sends $a \otimes b$ to
$\bar{a} \otimes \bar{b}$. We denote the Weyl equivalence class of $a
\otimes b$ by $[a \otimes b]$ and as $\overline{i}= -i$ the terms $[1
\otimes i]$ and $[i \otimes 1]$ vanish.

Frobenius reciprocity (FR) identifies
$[\res(x) \otimes b]$ with $x \otimes \tr(b)$ and $[a \otimes
\res(y)]$ with $\tr(a) \otimes y$ and therefore
$[1 \otimes 1] = 2 \cdot 1 \otimes 1$, so we only need to keep the
generators $1 \otimes 1$ and $[i \otimes i]$.

Assume that there is an isomorphism of Tambara functors $\phi \colon 
\und{\Z[i]}^\fix \Box \und{\Z[i]}^\fix \ra \und{(\Z[i] \otimes
  \Z[i])}^\fix$. Then at both levels we have an isomorphism of rings
and thus $1 \otimes 1$ has to map to $1 \otimes 1$ in both levels.

As $i \otimes i$ is a square root of $1 \otimes 1$, $\phi_e$ has to send it
again to such a square root. As we assume that $\phi_e$ is injective,
it cannot map it to $-1 \otimes 1$, so it has to sent it to $\pm i
\otimes i$. 
This results in
\begin{align*}
  \phi_{C_2}([i \otimes i]) & = \phi_{C_2}(\tr(i \otimes i)) \\
                          & =  \tr(\phi_e(i \otimes i)) \\
                          & = \pm \tr(i \otimes i) \\
                          & = \pm (i \otimes i + (-i) \otimes (-i)) =
                            \pm 2 \cdot i \otimes i.
                            \end{align*}
As $2$ is not invertible, this implies that the map $\phi_{C_2}$ is not
surjective. 
  \end{ex}

Later, in the context of Galois extensions,  we will be able to generalize Corollary \ref{cor:boxfix} to contexts where the group order is not necessarily invertible; see Remark \ref{rem:char2} (1).

\section{Hochschild homology and genuine K\"ahler differentials}  \label{sec:hh}

Mike Hill defined a notion of genuine K\"ahler differentials
\cite[Definition 5.4]{hill}. In \cite[Theorem 4.19]{mqs} the authors
relate them to the first equivariant Hochschild homology group for
Green functors. We first recall the definition of equivariant Hochschild homology:   
\begin{defn} \label{defn:hh} 
  Let $\und{R}$ be a commutative Green functor and let  $\und{R} \ra \und{T}$
  be a morphism of Green functors and let $\und{M}$ be a $\und{T}$-bimodule over $\und{R}$. The \emph{Hochschild homology of $\und{T}$ over $\und{R}$ with coefficients in $\und{M}$, $\GHH_*^{\und{R}}(\und{T}; \und{M})$}, is the homology of
  the chain complex associated with the simplicial Mackey functor whose part in  degree $q$ is 
  \[ C_q^\und{R}(\und{T}; \und{M}) = \und{M} \Box_\und{R} \und{T}^{\Box_\und{R} q}\]
  and whose differential is $b^{(q)} \colon C_q^\und{R}(\und{T}; \und{M}) \ra C_{q-1}^\und{R}(\und{T}; \und{M})$ with $b^{(q)} = \sum_{i=0}^q (-1)^ib_i$ where
  \[ b_i = \id^{\Box_\und{R} i-1} \Box_\und{R} \mu \Box_\und{R} \id^{\Box_\und{R} q-i-1} \text{ for } 0 \leq i < q\]
  Here $\mu$ denotes the left and right $\und{T}$-module structure of $\und{M}$ and the multiplication in $\und{T}$. The last face map $b_q$ mimics the last Hochschild face map, so
  \[ b_q = (\mu \Box_\und{R} \id^{\Box_\und{R} q-1}) \circ \gamma,\]
 where  $\gamma \colon \und{M} \Box_\und{R} \und{T}^{\Box_\und{R} q} \ra \und{T} \Box_\und{R} \und{M}
  \Box_\und{R} \und{T}^{\Box_\und{R} q-1}$ cyclically permutes the factors by bringing the last
  factor to the front. 
\end{defn}

The face maps give rise to a semi-simplicial structure. Together with
the degeneracies this gives a simplicial object in Mackey functors,
where the $i$th degeneracy map just inserts the unit $\eta \colon
\und{R} \ra \und{T}$ between the spot $i$ and $i+1$.

\begin{rem}
For a map of
$G$-Tambara functors 
$\und{R} \ra \und{T}$ this actually agrees with the equivariant Loday
construction with respect to $S^1$, $\cL_{S^1}^{G,
  \und{R}}(\und{T})$, as in  \cite[Definition 2.2]{lrz-loday} 
where $S^1 = \Delta_1/\partial\Delta_1$ is the small model of the simplicial
circle viewed as a finite $G$-simplicial set with trivial $G$-action. The
triviality of the $G$-action on $S^1$ and the fact that there is a
cyclic ordering on $S^1_n$ for all $n$ ensure that we 
can actually get away with fewer assumptions: We don't need
multiplicative norms and we do not need $\und{R}$ to be commutative.  

\end{rem}

In the non-equivariant commutative case the first Hochschild homology group also coincides with the module of K\"ahler differentials: For a commutative $k$-algebra $A$ and a symmetric $A$-bimodule $M$ over $k$: $\HH_1^k(A;M) \cong M \otimes_A \Omega^1_{A \mid k}$.

Recall from \cite[Definition 4.1]{hill} that a genuine $\und{R}$-derivation $d \colon \und{T} \ra \und{M}$ is a morphism of $G$-Mackey functors such that

\begin{enumerate}
\item
  $d$ satisfies a Leibniz rule: for all finite $G$-sets $S$ and all $a,b  \in \und{T}(S)$
  \[ d(ab) = ad(b) + d(a)b. \]
\item
  For all $a \in \und{T}(G/H)$ and for all $H < K < G$:
  \[ d(\norm_H^K(a)) = \tr_H^K(N_{\pi_2}R_{\pi_1}(a)\cdot d(a)). \]
  Here $\pi_i \colon G/H \times_{G/K} G/H \setminus \Delta \ra G/H$ is the projection to the $i$th factor and $\Delta$ is the diagonal. 
\item
  $d \circ \eta = 0$.   
\end{enumerate}
Conditions (1) and (3) agree with the non-equivariant conditions for a
derivation. Condition (2) can be thought of as a \emph{twisted Leibniz
  rule}. The genuine K\"ahler differentials are defined as
$\und{I}/\und{I}^{>1}$ where $\und{I}$ denotes the kernel of the
multiplication map $ \und{T} \Box_\und{R} \und{T} \ra \und{T}$ and
$\und{I}^{>1}$ is the ideal generated by the image of norms on
$2$-surjective maps. This contains the image of the multiplication map
on $\und{I}$.  The module of genuine K\"ahler differentials, 
$\Omega^{1,G}_{\und{T}^\fix/\und{R}^c}$, then 
represents genuine derivations \cite[Theorem 5.7]{hill}.

In the equivariant case, the relationship between the first Hochschild homology group and the module of genuine K\"ahler differentials is similar to the non-equivariant setting. A proof of the following result can also be found in \cite[Theorem 4.19]{mqs}. 

\begin{thm} \label{thm:hh1}
  If $\eta \colon \und{R} \ra \und{T}$ is a morphism of $G$-Tambara functors, then
  \[ \GHH_1^\und{R}(\und{T}) \cong \und{I}/\und{I}^2. \]

\end{thm}
The following proof is a direct transfer of the non-equivariant proof. 
\begin{proof}
  There is a canonical map $\varphi \colon \und{T} \ra \und{I}$ that takes the difference of
  the two maps
  \[ \xymatrix@1{\und{T} \cong \und{T} \Box_\und{R} \und{R} \ar[r]^{\id \Box \eta} & \und{T} \Box_\und{R} \und{T}}  \text{ and } \xymatrix@1{\und{T} \cong \und{R} \Box_\und{R} \und{T} \ar[r]^{\eta \Box \id} & \und{T} \Box_\und{R} \und{T}}. \]  We
  prolong this map with the projection and obtain
  $\und{T} \ra \und{I}/\und{I}^2$. As $\und{I}$ is a $\und{T}$-module (and so is $\und{I}/\und{I}^2$), this yields a morphism of Mackey functors
  \[ \phi \colon \und{T} \Box_\und{R} \und{T} \ra \und{T} \Box_\und{R} \und{I}/\und{I}^2 \ra \und{I}/\und{I}^2. \]
  We claim that $\phi$ factors through $\GHH_1^\und{R}(\und{T})$ and induces the
  desired isomorphism.

  We have to show that $\phi$ vanishes on the image of the Hochschild boundary
  \[ b^{(2)} = \mu \Box_\und{R} \id - \id \Box_\und{R} \mu + (\mu \Box_\und{R} \id) \circ \gamma. \]
  To that end we observe that the commutativity of $\und{T}$ ensures that we can use the permutation $\tau$ instead of $\gamma$, where $\tau$ interchanges the third box product factor with the second one. The resulting map is the same. But now we can rewrite $b^{(2)}$ as a composite
  \[ \xymatrix{
      \und{T} \Box_\und{R} \und{T} \Box_\und{R} \und{T} \ar[r]^{\id \Box \psi} & \und{T} \Box_\und{R} \und{R} \Box_\und{R} \und{T} \Box_\und{R} \und{T} \ar[d]^{\id \Box \eta \Box \id \Box \id} \\
&      \und{T} \Box_\und{R} \und{T} \Box_\und{R} \und{T} \Box_\und{R} \und{T} \ar[d]^{\id \Box b^{(2)}} \\
 &     \und{T} \Box_\und{R} \und{T} \Box_\und{R} \und{T} \ar[d]^{\mu \Box \id}\\
  &    \und{T} \Box_\und{R} \und{T}
    }\]
  where $\psi$ is the canonical isomorphism $\und{T} \Box_\und{R} \und{T} \cong \und{R} \Box_\und{R} \und{T} \Box_\und{R} \und{T}$.
The composite $b^{(2)} \circ (\eta \Box \id \Box \id)$ is in the image of $\und{I}^2$ because up to a degenerate summand it corresponds to the product $\mu \circ (\varphi \Box \varphi)$. 
  
\end{proof}

Using the identification $\und{M} \Box_{\und{R}} \und{T}^{\Box_\und{R} n} \cong \und{M} \Box_\und{T} (\und{T}^{\Box_\und{R} n+1})$ and an analogous proof as above
  we immediatly obtain the following: 
\begin{cor}
If $\eta \colon \und{R} \ra \und{T}$ is a morphism of $G$-Tambara
functors and if $\und{M}$ is a $\und{T}$-module, then
  \[ \GHH_1^\und{R}(\und{T};\und{M}) \cong \und{M} \Box_{\und{T}}
    \und{I}/\und{I}^2. \] 
\end{cor}

In the special case of fixed point Tambara functors, the classical
identification of the first Hochschild homology group in terms of the
module of K\"ahler differentials still holds: 
\begin{prop}
  Let $R$ be a commutative ring with trivial $G$-action and let $T$ be
  a commutative ring with $G$-action together with a map of
  commutative $G$-rings $R \ra T$. Then for every $\und{T}^\fix$-module
  $\und{M}$: 
  \[ \GHH_1^{\und{R}^c}(\und{T}^\fix; \und{M})  \cong \und{M}
    \Box_{\und{T}^\fix} \Omega^{1,G}_{\und{T}^\fix/\und{R}^c}. \] 
Here, $\Omega^{1,G}_{\und{T}^\fix/\und{R}^c}$ denotes the $\und{T}^\fix$-module of
genuine K\"ahler differentials \cite[Definition 5.4]{hill}. 
\end{prop}
Note that a ring map $R \ra T$ as above induces a map of $G$-Tambara functors
$\eta \colon \und{R}^c \ra \und{T}^\fix$.

\begin{proof}
It suffices to consider the case where $\und{M} = \und{T}^\fix$. For a standard projection $\pi_H^K \colon G/H \ra G/K$ the associated norm in a fixed point Tambara functor is given by 
  \begin{equation} \label{ex:norm-fix}
    \norm_H^K(a) = \prod_{g \text{ rep. for } K/H} ga. \end{equation}
  Hence in this case the twisted Leibniz rule follows from the ordinary Leibniz rule. 

\end{proof}


\begin{rem}
Nakaoka shows \cite[Proposition 4.21]{nakaoka} that a $G$-Tambara
functor $\und{B}$ is a sub-Tambara functor of a Tambara functor of the
form $\und{T}^\fix$ if and only if $\und{B}$ has injective restriction
maps for all maps between transitive $G$-sets and he calls this
property (MRC). In sub-Tambara functors of fixed point Tambara
functors, we again have the formula for the norm as in
\ref{ex:norm-fix}, so here again, the twisted Leibniz rule follows
from the ordinary one. 

In particular, a cohomological $G$-Tambara functor $\und{T}$ with the
property that for all $H < G$ the commutative ring $\und{T}(G/H)$ has
no $|G|$-torsion satisfies this condition \cite[Example
4.20]{nakaoka}. Nakaoka also shows \cite[Theorem 4.32]{nakaoka} that
field-like Tambara functors satisfy (MRC). This class of Tambara
functors is extensively studied in \cite{wisdom}. See also Section
\ref{sec:nullstellensatzian} below. 
\end{rem}

  \section{\'Etale and separable Tambara functors} \label{sec:etaleseparable}
Recall from \cite{hill} that for a map of $G$-Tambara functors
$\und{R} \ra \und{T}$ we call $\und{T}$ \emph{formally \'etale} over
$\und{R}$ if $\und{T}$ is flat as an $\und{R}$-module and if
$\Omega^{1,G}_{\und{T}/\und{R}} = 0$. In the following we relate the
notion of separability to formal \'etaleness. As the definition of
separability is only based on the module category, we can define it as usual:
\begin{defn} 
  Let $\und{R} \ra\und{T}$ be a map of $G$-Tambara functors. Then $\und{T}$ is
  \emph{separable over} $\und{R}$ if the multiplication map of $\und{T}$,
  $\mu \colon \und{T} \Box_{\und{R}} \und{T} \ra \und{T}$ has a section $s \colon \und{T} \ra \und{T} \Box_{\und{R}} \und{T}$ which is a map of $\und{T}$-bimodules over $\und{R}$. 
\end{defn}  

\begin{ex}
  If a commutative ring $T$ is separable over a commutative ring $R$,
  then the constant $G$-Tambara functor $\und{T}^c$ is separable over
  $\und{R}^c$. This follows because $\und{T}^c \Box_{\und{R}^c}
  \und{T}^c \cong \und{(T \otimes_R T)}^c$ (see \cite[Lemma 5.1]{lrz}). 
  \end{ex}

\begin{ex} \label{ex:fixsep}
  Let $R \ra T$ be a map of commutative rings and assume that a group $G$ acts trivially on $R$ and acts on
  $T$ by $R$-algebra maps. If $|G|$ is invertible in $T$ and if $T$ is separable over $R$, then Corollary \ref{cor:boxfix} implies that $\und{T}^\fix$ is separable over $\und{R}^c$. 

\end{ex}

\begin{prop} \label{prop:sephh}
  Let  $\und{R} \ra \und{T}$ be a map of Tambara functors and assume that
  $\und{T}$ is separable and flat over $\und{R}$. Then the canonical map
  $\und{T} \ra \GHH_*^\und{R}(\und{T})$ is an isomorphism.  
\end{prop}

\begin{proof}
  Thanks to flatness we have 
  \[\GHH_*^\und{R}(\und{T}) \cong \Tor_*^{\und{T} \Box_\und{R} \und{T}}(\und{T}, \und{T}) \]
  but the separability of $\und{T}$ over $\und{R}$ ensures that $\und{T}$ is a
  projective $\und{T} \Box_\und{R} \und{T}$-module. 
\end{proof}

As there is a surjection $\und{I}/\und{I}^2 \ra \und{I}/\und{I}^{>1}$, Theorem \ref{thm:hh1} implies the following fact: 
\begin{cor}
  If $\und{T}$ is separable and flat over $\und{R}$, then $\und{T}$ is
  formally \'etale over $\und{R}$.
  \end{cor}

  We now study how restrictions and norms interact with separability. Recall that for a subgroup $H < G$, the restriction functor 
  \[ i_H^G \colon G\tamb \ra H\tamb \]
  defined as  $i_H^G\und{R}(H/K) := \und{R}(G/H \times_H H/K) = \und{R}(G/K)$ has both a left adjoint (which is the norm functor $N_H^G$) and a right adjoint
  (which is co-induction). 

  \begin{prop}
    \begin{enumerate}
    \item[]
      \item 
    Restriction preserves separability.
\item
  If $\und{R} \ra \und{T}$ is separable, then so is $N_H^G\und{R} \ra N_H^G\und{T}$.
  \end{enumerate}
  \end{prop}  

\begin{proof}

  Assume that $\und{R} \ra \und{T}$ is separable in the category of $G$-Tambara
  functors. 
  Let $s \colon \und{T} \ra \und{T} \Box_{\und{R}} \und{T}$ be a
  $\und{T}$-bimodule section of the multiplication map of $\und{T}$. As $i_H^G$ is strong symmetric monoidal and as it preserves coequalizers because it is a left adjoint, the relative box product $i_H^G(\und{T}) \Box_{i_H^G(\und{R})} i_H^G(\und{T})$ is isomorphic to $i_H^G(\und{T} \Box_{\und{R}} \und{T})$ and
  \[ i_H^G(s) \colon i_H^G\und{T} \ra i_H^G(\und{T} \Box_{\und{R}} \und{T}) \]
  is a section of the multiplication map on $i_H^*\und{T}$. The fact that $i_H^*$ is strong symmetric monoidal then also ensures that $i_H^*(s)$ is a map of $i_H^*(\und{T})$-bimodules.

  As the norm is a left adjoint and is strong symmetric monoidal, the analogue of the above proof applied to a map $\und{R} \ra \und{T}$ of $H$-Tambara functors  also shows that $N_H^G(\und{R}) \ra N_H^G(\und{T})$ is separable if $\und{R} \ra \und{T}$ was separable.





\end{proof}

  \section{Separability and Loday constructions} \label{sec:appl-loday}

In the non-equivariant setting, algebraic homology theories like
Hochschild homology interpret coeffients to be located at the base
point: For Hochschild homology of an associative $k$-algebra $A$ one
would place an $A$-bimodule $M$ at the base point of the simplicial
model of the $1$-sphere, $S^1$, and copies of $A$ at the other points of $S^1$. In the equivariant setting for a
finite group $G$ the trivial
orbit $G/G$ is the terminal object in the category of finite $G$-sets
and $G$-equivariant maps, but note that $G/G$ is \emph{not} an initial object.

There are natural examples of equivariant homology
theories where one might want to glue the coefficients to a non-trivial
orbit. For instance for the group or order two there are three naturally occurring
finite
$C_2$-simplicial sets, whose underlying simplicial set models $S^1$:
We can consider $S^1 = \Delta_1/\partial\Delta_1$ with the trivial
$C_2$-action, there is the simplicial model of the one-point
compactification of the real sign-representation, $S^\sigma$, and
there is a rotation circle $S(\pi)$ for the rotation by $180$ degrees which is the unit sphere in the $2$-dimensional rotation representation on $\mathbb{R}^2$:

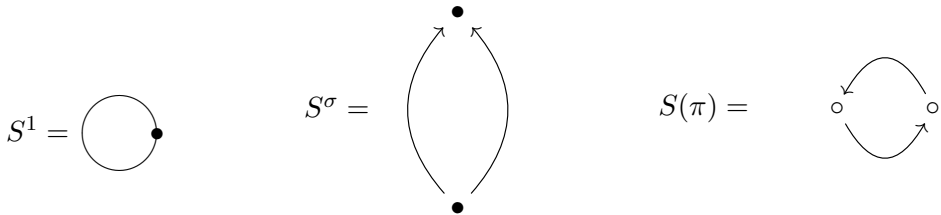
\vspace{1cm}

\begin{picture}(2,2)
\setlength{\unitlength}{1cm}
\put(1,-1.5){\circle{1}}
\put(1.4,-1.6){$\bullet$}
\put(-0.5,-1.6){$S^1=$}
\end{picture}
\hspace{3cm}
$\xymatrix{& \bullet & \\ S^\sigma = & & \\ & \bullet
  \ar@/^4ex/[uu]\ar@/_4ex/[uu]& }$
\hspace{1cm} $\xymatrix{& & \\
  S(\pi) = &  \circ \ar@/_4ex/[r] & \circ  \ar@/_4ex/[l]\\
 & & }$

Note that $S^\sigma$ is equivalent to the Segal-Quillen subdivision
\cite[Appendix 1]{segal} of $S^1$. 
In the first two examples the zero simplices carry a trivial
$C_2$-action, whereas in the third example they give rise to a free
orbit. Following the definition of equivariant Loday constructions from
\cite{lrz-loday} it is therefore natural to associate to a Tambara
functor $\und{R}$ and to a symmetric $\und{R}$-bimodule
$\und{M}$  the Loday construction which has as zero simplices 
\[\cL_{S(\pi)}^{C_2}(\und{R}; \und{M})_0 = N_e^{C_2}i_e^*\und{M}. \]
Here, $N_e^{C_2}$ denotes the norm functor $N_e^{C_2} \colon e\text{-}\mack
\ra C_2\text{-}\mack$, constructed for instance in \cite{mazur,hoyer}. 

For $S^1$ with a trivial $C_2$-action we get the equivariant
version of Hochschild homology studied above (compare Definition
\ref{defn:hh}) where the tensor products 
are replaced 
by the box product. In particular, 
\[\cL_{S^1}^{C_2}(\und{R}; \und{M})_0 = \und{M}. \]

As $S^\sigma$ is the Segal-Quillen subdivision of $S^1$ with
$S^\sigma_0 = sq(S^1)_0 = S^1_1$ one might place $\und{M}$ at one of
the trivial orbits and $\und{R}$ at the other: 
\[ (\cL_{S^\sigma}^{C_2}(\und{R}; \und{M}))_0 =
  \cL_{S^1}^{C_2}(\und{R}; \und{M}))_1 = \und{M} \Box \und{R}. \] 
This is the convention that we used in \cite{lr}.

In the non-equivariant context if a  map $A \ra B$ is \'etale, then
the canonical map from $B$ to the Hochschild homology of $B$ over $A$,
$B \ra \HH_*^A(B)$, is an isomorphism \cite{weibelgeller}. One can
view this map as the map that sends $B$ located at a point to $B$
located at the basepoint of $\Delta_1/\partial\Delta_1$. If we aim at
a corresponding result in equivariant algebra, then we might have to
adapt this to several flavors of basepoints. So one might expect that
in the case of $S(\pi)$ a formally  \'etale map $\und{R} \ra \und{T}$ of $G$-Tambara functors in the sense of Hill \cite{hill} induces an isomorphism on homotopy groups 
\[ N_e^{C_2,\und{R}}i_e^{C_2}\und{T} \ra \cL_{S(\pi)}^{C_2,\und{R}}(\und{T})\]
where the left hand side is the constant simplicial Tambara functor on the relative norm restriction: 
\[  N_e^{C_2,\und{R}}i_e^{C_2}\und{T} = N_e^{C_2}i_e^{C_2,}\und{T} \Box_{N_e^{C_2}i_e^{C_2}\und{R}} \und{R}. \]

We know by Proposition \ref{prop:sephh} that separability ensures the
vanishing of equivariant  
Hochschild homology groups in positive degrees. We show that this
  also holds for the Loday construction on the flip circle:
\begin{prop} \label{prop:flip-separable}

  Assume that $\und{R} \ra \und{T}$ is a map of $C_2$-Tambara functors
that turns $\und{T}$ into a separable $\und{R}$-algebra and assume
that $\und{T}(C_2/e) =: T$ is flat as an $\und{R}(C_2/e)$-module. Then
\[\und{\pi}_i\cL_{S^\sigma}^{C_2,\und{R}}(\und{T})
  \cong 0 \text{ for all } i > 0. \]
\end{prop}

\begin{proof}

  For the Loday construction on $S^\sigma$  we show in \cite[Proposition 3.4]{lrz-realhh} that it doesn't matter which Weyl group action one uses on the norm-restriction terms. For the proof we use the diagonal action. We first show that
  \[ \pi_i\cL_{S^\sigma}^{C_2,\und{R}}(\und{T})(C_2/e)  \cong
    \begin{cases} T, & i = 0, \\ 0, & i > 0. \end{cases}\]
\noindent
At the free level, the relative norm-restriction $N_e^{C_2,
  \und{R}}i_e^{C_2}(\und{T})$ is just
\[ N_e^{C_2,
    \und{R}}i_e^{C_2}(\und{T})(C_2/e) \cong T \otimes_R T. \]
where $R = \und{R}(C_2/e)$. As $\und{T}$ is
separable over $\und{R}$, the commutative ring $T$ is separable over
$R$. The $N_e^{C_2,
    \und{R}}i_e^{C_2}(\und{T})(C_2/e)$-module structure of $T$ is just
  given by the augmentation map followed by the multiplication in $T$
  and hence we can identify the chain complex associated with the
  simplical object $\cL_{S^\sigma}^{C_2,\und{R}}(\und{T})(C_2/e)$ with
  the two-sided bar construction $B_*^R(T, T\otimes_R T, T)$ and as
  $T$ is flat over $R$, its
  homology groups calculate $\Tor_i^{T \otimes_R T}(T,T)$. The
  separability of $T$ over $R$ ensures that $T$ is projective as a $T
  \otimes_R T$-module, so the homology groups in positive degrees 
  vanish and in degree zero we obtain $T \otimes_{T \otimes_R T} T
  \cong T$.

As the homotopy groups
$\und{\pi}_*\cL_{S^\sigma}^{C_2,\und{R}}(\und{T})$ form a graded
Tambara functor, the triviality of the homotopy groups 
$\pi_i\cL_{S^\sigma}^{C_2,\und{R}}(\und{T})(C_2/e)$ for positive $i$
ensure the triviality of
$\pi_i\cL_{S^\sigma}^{C_2,\und{R}}(\und{T})(C_2/C_2)$ for positive
$i$: the norm map sends $1=0$ in the free level to $1$ in the
$C_2/C_2$-level, so we also get  $1=0$ at $C_2/C_2$. 

\end{proof}

Note that the constant simplicial Tambara functor with value $\und{T}$
splits off $\cL_{S^\sigma}^{C_2,\und{R}}(\und{T})$: There are two
fixed points in $S^\sigma$ that correspond to the $0$ in the real sign
representation and the compactification point $\infty \in S^\sigma$,
so we get two inclusions of the constant simplicial $C_2$-set that has
$C_2/C_2$ in every simplicial degree into
$\cL_{S^\sigma}^{C_2,\und{R}}(\und{T})$. As $C_2/C_2$ is the terminal
object in the category of finite $C_2$-sets, we also get a canonical
projection map $S^\sigma_n \ra C_2/C_2$ that assembles into a map from
$S^\sigma$ to the constant simplicial finite $C_2$-set with value
$C_2/C_2$. On the level of Loday constructions this splits off a
constant $\und{T}$.  

In the above result we couldn't really pin down that
$\und{\pi}_0\cL_{S^\sigma}^{C_2,\und{R}}(\und{T})(C_2/C_2)$ agrees
with $\und{T}(C_2/C_2)$. We know that we get at least a summand
$\und{T}$ in $\und{\pi}_0\cL_{S^\sigma}^{C_2,\und{R}}(\und{T})$, but a
priori it could happen that
$\und{\pi}_0\cL_{S^\sigma}^{C_2,\und{R}}(\und{T})(C_2/C_2) \cong
\und{T}(C_2/C_2) \oplus M$ for some non-trivial $M$. 
However, for fixed point Tambara functors we can rule that out: 

\begin{cor} \label{cor:flipfix-separable}
Let $R \ra T$ be a map of commutative $C_2$-rings where $R$ has trivial $C_2$-action, such that $\und{T}^\fix$ is separable and flat over $\und{R}^c$. Then 
\[ \und{\pi}_0\cL_{S^\sigma}^{C_2,\und{R}^c}(\und{T}^\fix) \cong \und{T}^\fix \]
as Tambara functors. 
\end{cor}  

\begin{proof}
  In general, the restriction functor $i_e^{C_2} $ from the category
  of $C_2$-Tambara functors to $e$-Tambara functors (\ie, to
  commutative rings) is not faithful. However, if we restrict
  $i_e^{C_2}$ to the subcategory of fixed point Tambara functors then
  it \emph{is} faithful, because the behaviour at the free level
  determines the behaviour at $C_2/C_2$ in these cases. Thus we know that the counit
  map $N_e^{C_2}i_e^{C_2}\und{T}^\fix \ra \und{T}^\fix$ is surjective. As the counit induces a well defined map on the relative norm-restriction
\[ N_e^{C_2}i_e^{C_2}\und{T}^\fix  \Box_{N_e^{C_2}i_e^{C_2} \und{R}^\fix} \und{R}^\fix \ra \und{T}^\fix\]  
  this implies that the augmentation map on the relative norm 
  is also surjective and hence we can calculate $\und{\pi}_0$ of the Loday
  construction for $S^\sigma$ as the cokernel of the map
  \[ \mu \Box \id - \id \Box \mu \colon \und{T}^\fix \Box_{\und{R}^c}
    \und{T}^\fix \Box_{\und{R}^c} \und{T}^\fix \ra
    \und{T}^\fix \Box_{\und{R}^c} \und{T}^\fix\]
  which is isomorphic to $\und{T}^\fix$. 

\end{proof}

We relate the vanishing result above to a vanishing result of
reflexive homology, $\HR^+$.

\begin{cor}
Assume that $T$ is a commutative and separable $R$-algebra with a
$C_2$-action by $R$-algebra maps. If  $2$ is invertible in $R$ and if $T$ is flat over $R$, then
\[ \HR^{+,R}_i(T) \cong \begin{cases} T^{C_2}, & i=0, \\ 0, & i > 0. \end{cases} \]
\end{cor}
\begin{proof}  
  We actually give two proofs:

By Example \ref{ex:fixsep} we obtain that $\und{T}^\fix$ is a commutative and  separable $\und{R}^c$-algebra Tambara functor.   
  By \cite[Theorem 6.5]{lr} we know that under the above assumptions 
\[ \HR^{+,R}_i(T) \cong
  \und{\pi}_i\cL_{S^\sigma}^{C_2,\und{R}^c}(\und{T}^\fix)(C_2/C_2)\]
and therefore Corollary \ref{cor:flipfix-separable} implies the
claim. 

The second proof argues directly with reflexive homology: Daniel
Graves developed a spectral sequence converging to $\HR^{+,R}(T)$ in 
  \cite{graves} by constructing a bicomplex whose total complex has
  $\HR^{+,R}(T)$ as its homology groups. Calculating the vertical
  homology of this bicomplex results in the Hochschild homology groups
  $\HH_*^R(T)$ in every column.  The horizontal homology then results
  in the homology groups $H_*(C_2;\HH_*^R(T))$ of the group $C_2$ with
  coefficients in $\HH_*^R(T)$. By assumption the groups $\HH_*^R(T)$
  are concentrated in the zero line with value $\HH_0^R(T) = T$. As $2$ is
  invertible in $R$, the group homology is trivial in positive
  degrees. Thus in total we obtain 
\[ \HR_n^{+,R}(T) \cong \begin{cases} T_{C_2}, & n = 0, \\ 0,
                                                        &
                                                          n>0. \end{cases} \] 
  Here, $T_{C_2}$ denotes the $C_2$-coinvariants of $T$ and in our
  case these are isomorphic to the $C_2$-invariants. 
\end{proof}
\begin{rem}
  In the non-equivariant context formal \'etaleness of a commutative algebra
  $A$ over $k$ implies that there is a weak equivalence  $A \simeq
  \cL_X^k(A)$ and in particular $\pi_0\cL_X^k(A) \cong A$ and $\pi_{i>0}\cL_X^k(A) = 0$ for all connected finite simplicial sets $X$ \cite[Proposition 2.11]{lr-stable}.  
There are equivariant homology theories for which this result does not carry over, as the following example shows.  
\end{rem}

\begin{ex}
We claim that the zeroth homology of the twisted cyclic nerve of
\cite[Definition 2.20]{bghl} vanishes at the free level 
for the separable extension $\und{\Q}^c \rightarrow
\und{\Q(i)}^\fix$. The simplicial structure of the twisted cyclic
nerve almost agrees with the one of equivariant Hochschild homology,
but the last face map is twisted by the group action. This homology theory is therefore modelled on $S^1$. 

For the proof we show by brute force that the differential 
\[ d^\tau \colon \und{\Q(i)}^\fix \Box_{\und{\Q}^c} \und{\Q(i)}^\fix \ra
  \und{\Q(i)}^\fix\] is surjective at the free level. 

We can write down the differential explicitly as 
\begin{align*}
  d^\tau \colon \Q(i) \otimes_\Q \Q(i) & \ra \Q(i), \\
  d^\tau((a+bi) \otimes (u + vi)) & = (a+bi)(u+iv) - (u-vi)(a+bi), 
\end{align*}
and this is $-2bv + 2avi$. A general element $x +yi$ is then hit by
the differential for arbitrary $u$, $v=1$, $b = -\frac{x}{2}$ and $a =
\frac{y}{2}$.

\end{ex}

\section{Galois extensions of Tambara functors} \label{sec:galois}
In non-equivariant algebra Galois extensions are an important source
of separable and \'etale extensions. We study the corresponding notion for Tambara
functors. Here, the definition is based on Galois extensions of commutative
rings, developed originally by Auslander and Goldman \cite[Appendix]{auslandergoldman}. 

Before we state our definition, we note the following fact: 

\begin{lem}
  Assume that $\und{R} \ra \und{T}$ is a morphism of $G$-Tambara
  functors. 
If a finite group $H$ acts trivially on $\und{R}$ and acts on the left on $\und{T}$ through
$\und{R}$-algebra maps in $G$-Tambara functors, then the fixed points
$\und{T}^H$ defined levelwise as $\und{T}^H(G/K) = \und{T}(G/K)^H$
constitute a sub-$G$-Tambara functor of $\und{T}$ and $\und{T}^H$ is
still a commutative $\und{R}$-algebra. 
\end{lem}  

\begin{proof}
As $H$ acts via maps of $G$-Tambara functors, the structure maps of
$\und{T}$ commute with the $H$-action. Therefore the fixed points
inherit a $G$-Tambara structure. As $H$ fixes $\und{R}$, the fixed
points are still a commutative $\und{R}$-algebra. 
\end{proof}

For a map of $G$-Tambara functors $\varphi \colon \und{R} \ra \und{T}$ we
denote the multiplication map of $\und{T}$ over  $\und{R}$  by
$\mu \colon \und{T} \Box_{\und{R}} \und{T} \ra \und{T}$. 

\begin{defn}
Let $H$ and $G$ be finite groups. 
Let $\varphi \colon \und{R} \ra \und{T}$ be a levelwise injective morphism of 
$G$-Tambara functors. We call  $\varphi \colon \und{R} \ra \und{T}$ an 
\emph{$H$-Galois extension of $G$-Tambara functors} if the following 
conditions hold:

\begin{enumerate}
\item
The group $H$ acts trivially on $\und{R}$ and on the left on $\und{T}$ through
$\und{R}$-algebra maps in $G$-Tambara functors.

\item
The morphism  $\varphi \colon \und{R} \ra \und{T}$
induces an isomorphism 
\[ \varphi \colon \und{R} \cong \und{T}^{H},  \]
\ie, for all $K < G$
\[ \varphi_K \colon \und{R}(G/K) \cong \und{T}(G/K)^{H}.\]
\item
  The morphism  
  \[ \chi \colon \und{T} \Box_{\und{R}} \und{T} \ra
    \prod_H  \und{T}\]
  is an isomorphism of $G$-Tambara functors, where $\chi$ is defined as
  \[ \xymatrix{ \und{T} \Box_{\und{R}} \und{T} \ar@/_4ex/[rrrr]_{\chi}
      \ar[rrr]^{(\id \Box h)_{h \in H}}
      &&& \prod_H \und{T} \Box_{\und{R}} \und{T} \ar[r]^{\prod_H \mu} & \prod_H \und{T} }\] 
\end{enumerate}
\end{defn}
  
\begin{rem}
  If $\varphi \colon \und{R} \ra \und{T}$ is an $H$-Galois extension of
  $G$-Tambara functors, then the underlying extension of commutative rings
  \[ \varphi_e \colon R = \und{R}(G/e) \ra \und{T}(G/e) = T\]
  is an $H$-Galois extension of commutative rings because 
$\varphi_e\colon R \rightarrow T$ is a injective map of commutative rings, so 
that $H$ acts on
  $T$ by $R$-algebra maps, fixing $R$, such that $\varphi_e \colon R \cong
  T^H$ and $\chi_e \colon (\und{T} \Box_{\und{R}} \und{T})(G/e) = T \otimes_R T
  \ra \prod_H T$ is an isomorphism, given by $t_1 \otimes t_2 \mapsto
  (t_1 h(t_2))_{h \in H}$, thus $\varphi_e$ satisfies the conditions of
  \cite[Definition 1.4]{chr}.  
\end{rem}

We can also transform every Galois extension of commutative rings into a
Galois extensions of $G$-Tambara functors via the constant Tambara functor:
\begin{prop}  \label{prop:constantgalois}
Assume that $R \ra T$ is an $H$-Galois extension of commutative rings. Then $\und{R}^c \ra \und{T}^c$ is an $H$-Galois extension of $G$-Tambara functors for every finite group $G$ and $\und{T}^c$ is flat as an $\und{R}^c$-module. 
\end{prop}  

\begin{proof}
  We let $H$ act on $\und{T}$ by acting with $H$ on every level
  $\und{T}^c(G/K) = T$. As the Weyl action on $\und{R}^c$ and $\und{T}^c$ is
  trivial, it commutes with the $H$-action on $\und{T}^c$. As the restriction maps are the identity, they also commute with the $H$-action. The  transfer
  just multiplies by the index and the norm takes an element to the power of that element by the index, thus they also commute with the $H$-action. The $H$-fixed points in every level are precisely $R$. 

  Finally, as $\und{T}^c \Box_{\und{R}^c} \und{T}^c \cong \und{(T \otimes_R T)}^c$ by \cite[Lemma 5.1]{lrz}, we get that
  \[ \und{T}^c \Box_{\und{R}^c} \und{T}^c \cong \und{(T \otimes_R T)}^c \cong \und{\prod_H T}^c. \]
      As constant Tambara functors are fixed point Tambara functors and as the latter are right adjoints, they commute with products, so in total we obtain
 \[ \und{T}^c \Box_{\und{R}^c} \und{T}^c \cong \prod_H \und{T}^c. \]

 It remains to show that $\und{T}^c$ is flat over $\und{R}^c$. To that end we prove that for every $\und{R}^c$-module $\und{M}$ we have at every level $G/K$ that
 \begin{equation} \label{eq:tensor-constant}
   \und{M} \Box_{{R}^c} \und{T}^c(G/K) \cong \und{M}(G/K) \otimes_R T.\end{equation}
 This proves the claim because for every $H$-Galois extension $R \ra T$ of commutative rings, $T$ is flat (even projective) as an $R$-module \cite[Theorem 1.3]{chr}. 

 In order to show \eqref{eq:tensor-constant} we use \cite[Proposition 3.14]{strickland} where Strickland  shows that we can represent every element in the box product $\und{M} \Box \und{T}^c(G/K)$ as $T_p(m \otimes t)$ with 
 \[ m \otimes t \in \und{M}(G/L) \otimes \und{T}^c(G/L) = \und{M}(G/L) \otimes T\]
 and where $p \colon G/L \ra G/K$ a map of finite $G$-sets. With the help of Frobenius reciprocity \cite[Lemma 3.13]{strickland}
 and by writing every $t \in \und{T}^c(G/L)$ as $R_p(t) = \id(t) = t$ we can rewrite
 $T_p(m \otimes t)$ as
 \[ T_p(m \otimes t) = T_p(m \otimes R_p(t)) = T_p(m) \otimes t.\]
 Therefore we can push every representative in $\und{M} \Box \und{T}^c(G/K)$ to
 an element $T_p(m) \otimes t \in \und{M}(G/K) \otimes T$. Using that $\und{M} \Box_{\und{R}^c} \und{T}^c$ is the coqualizer of a diagram involving
 \[ \und{M} \Box \und{R}^c \Box  \und{T}^c \cong \und{M} \Box \und{(R \otimes T)}^c\]
 and $\und{M} \Box \und{T}^c$,  we get the claim. 
 
\end{proof}

As for commutative rings we always get non-connected examples of
Galois extensions (see for
instance \cite[Example (b) in \S 1]{greither}):
\begin{ex}
  If $\und{R}$ is an arbitrary $G$-Tambara functor for any finite group $G$ and if $H$ is a finite group, then
  \[ \und{R} \ra \prod_H \und{R}\]
  is an $H$-Galois extension. Here, $H$ acts on the indexing set of the product and $\und{R}$ embeds into  $\prod_H \und{R}$ as a diagonal copy. As $H$ is finite, $\prod_H \und{R}$ is free over $\und{R}$; in particular, it is flat. 

\end{ex}

The next result is a key property of Galois extensions: 
\begin{thm} \label{thm:galoisseparable}
If $\varphi \colon \und{R} \ra \und{T}$ is an $H$-Galois extension of
$G$-Tambara functors, then $\und{R}  \ra  \und{T}$ is separable. 
\end{thm}
\begin{proof}
We define $\sigma \colon \und{T} \ra \und{T} \Box_{\und{R}} \und{T}$
as $\chi^{-1} \circ i^{(e)}$ where $i^{(e)} \colon \und{T} \ra \prod_H
\und{T}$ is the inclusion into the factor corresponding to the neutral
element $e \in H$. Note that $i^{(e)}$ does not preserve the
multiplicative unit, but it is a map of Mackey functors and preserves
the multiplication and commutes with the norm.

We have to show that $\mu \circ \sigma = \id_\und{T}$ and that
$\sigma$ is a $\und{T}$-bimodule map relative $\und{R}$. 
For the first property, consider the following commutative diagram:

\bigskip

\[ \xymatrix{ \und{T} \ar@{=}[drr]\ar[rr]^{i^{(e)}}  \ar@/^6ex/[rrrrr]^\sigma& &\prod_H \und{T}
    \ar[d]^{p^{(e)}} 
    & \prod_H \und{T} \Box_{\und{R}} \und{T} \ar[d]^{p^{(e)}} \ar[l]_(0.55){\prod_H \mu} 
    & &\und{T} \Box_{\und{R}} \und{T} \ar[ll]_{(\id \Box h)_{h \in H}} \ar@{=}[dll]\\
&  & \und{T} & \und{T} \Box_{\und{R}} \und{T} \ar[l]^\mu& &} \]
where $p^{(e)}$ is the projection to the $e$-coordinate. This shows
$\mu \circ \sigma = p^{(e)} \circ i^{(e)} = \id_{\und{T}}$.

For the compatibility with the bimodule structure, we adapt the
classical proof, see for instance \cite[Proof of Lemma 2]{takeuchi}, to the equivariant setting: 
Note that the map $\chi \colon \und{T} \Box_{\und{R}} \und{T} \ra
\prod_H \und{T}$ is equivariant if we use the $H$-action on the right
factor of $\und{T}$ in $\und{T} \Box_{\und{R}} \und{T}$ and the action
on $\prod_H \und{T}$ on the indexing set by right multiplication. 
For
every $h \in H$ composing $i^{(e)}$ with the action by $h$ gives the
inclusion into the factor with label $h^{-1}$, $i^{h^{-1}}$. Therefore
the following diagram commutes:
\[ \xymatrix{
& \prod_H \und{T} \ar[r]^h & \prod_H \und{T} \ar[dr]^{p^{(e)}}& \\
\und{T} \ar[r]^\sigma \ar[ur]^{i^{(e)}}
    \ar@/^10ex/[urr]^{i^{(h^{-1})}} & \und{T}
    \Box_{\und{R}} \und{T} \ar[r]^{\id \Box h} \ar[u]^\chi& \und{T}
    \Box_{\und{R}} \und{T} \ar[u]^\chi \ar[r]^\mu & \und{T}.}\]
Here, we have used that the multiplication $\mu$ coincides with the
composite $p^{(e)} \circ \chi$. 

With $\delta_{h,e}$ denoting the Kronecker delta for $h$
and $e$ this yields:  
\begin{equation} \label{eq:takeuchi}
\mu \circ (\id \Box h) \circ \sigma = p^{(e)} \circ i^{(h^{-1})} =
\delta_{h,e}\cdot \id_\und{T}. 
\end{equation}
In order to avoid a glaring clash of notation we denote by $s_H$ what is usually called the norm or trace of the finite
group $H$: \begin{equation} \label{eq:trace} s_H = \sum_{h \in H} h. \end{equation} 
With \eqref{eq:takeuchi} we can rewrite the identity map on $\und{T}$
as follows: 

\begin{equation}  \label{eq:identity1}
  \id_\und{T} = \mu \circ (\id \Box s_H) \circ \sigma. 
\end{equation}
With a similar argument one shows
\begin{equation}  \label{eq:identity2}
  \id_\und{T} = \mu \circ (s_H \Box \id) \circ \sigma. 
\end{equation}

We construct two auxiliary maps from $\und{T}$ to $\und{T}
\Box_\und{R} \und{T}$. Let $\eta \colon \und{R} \ra \und{T}$ denote
the unit of $\und{T}$ and let $\nu_L \colon \und{T} \cong \und{R}
\Box_{\und{R}}\und{T}$ denote the monoidal isomorphism. We let
$\psi_L$ denote the composite  
\begin{equation} \label{eq:psil}
\xymatrix@1{\und{T} \ar[r]^(0.4)\cong_(0.4){\nu_L} & \und{R} \Box_{\und{R}}
  \und{T} \ar[r]^{\eta \Box \id} & \und{T} \Box_{\und{R}} \und{T}
  \ar[r]^(0.4){\sigma \Box \id} & \und{T} \Box_{\und{R}}\und{T}
  \Box_{\und{R}}\und{T} \ar[r]^(0.6){\id \Box \mu} & \und{T} \Box_{\und{R}} \und{T}}
\end{equation}
and dually, let $\psi_R$ be the composite
\begin{equation} \label{eq:psir}
\xymatrix@1{\und{T} \ar[r]^(0.4)\cong_(0.4){\nu_R} & \und{T} \Box_{\und{R}}
  \und{R} \ar[r]^{\id \Box \eta} & \und{T} \Box_{\und{R}} \und{T}
  \ar[r]^(0.4){\id \Box \sigma} & \und{T} \Box_{\und{R}}\und{T}
  \Box_{\und{R}}\und{T} \ar[r]^(0.6){\mu \Box \id} & \und{T} \Box_{\und{R}} \und{T}}
\end{equation}
In the following we abbreviate $\sigma \circ \eta$ by $\xi$ so that we
can write
\[ \psi_L = (\id \Box \mu) \circ (\xi \Box \id) \circ \nu_L \text{ and
  } \psi_R = (\mu \Box \id) \circ (\id \Box \xi) \circ \nu_R. \]

We claim that $\psi_L = \psi_R = \sigma$ which ensures that $\sigma$
is a bimodule map. In the following diagrams we omit the unit isomorphisms
$\nu_R$ and $\nu_L$, unadorned box products are over $\und{R}$, $\id_n$ denotes an identity map on $n$ box product factors and
$\mu_n$ denotes an iterated multiplation map.

For the claim that $\psi_L$ and $\psi_R$ agree we consider the
composite
\[ \theta = (\xymatrix@1{\und{T} \ar[rr]^{\xi \Box \id \Box \xi} & &
    \und{T}^{\Box 5} \ar[rr]^{\id \Box \mu_3 \Box \id} & &
    \und{T}^{\Box 3} \ar[rr]^{\id \Box s_H \Box \id} & & \und{T}^{\Box 3} \ar@/^4ex/[r]^{\mu \Box
      \id} \ar@/_4ex/[r]_{\id \Box \mu}& \und{T} \Box_{\und{R}}  \und{T}})\]
Note that for defining $\theta$ it is irrelevant whether we use $\mu \Box \id$ or $\id \Box \mu$ as the last map in the composite because the image of $s_H$ is contained in
$\und{R}$. We compare $\psi_R$ with the $\theta$
using the map $\mu \Box \id$ at
the end of the composition and leave the dual proof of comparing $\psi_L$ with the other variant to the diligent reader.

We embed the above composite $\theta$ as the top row of  the following enlarged diagram 
\[ \xymatrix{\und{T} \ar[rr]^{\xi \Box \id \Box \xi} \ar[d]_{\id \Box
    \xi} & &
    \und{T}^{\Box 5} \ar[rr]^{\id \Box \mu_3 \Box \id} \ar[d]_{\id_2
\Box \mu \Box \id}& &
    \und{T}^{\Box 3} \ar[rr]^{\id \Box s_H \Box \id} & & \und{T}^{\Box
      3} \ar[r]^{\mu \Box \id} & \und{T} \Box_{\und{R}} \und{T} \\
 \und{T} \Box_{\und{R}}  \und{T} \Box_{\und{R}}  \und{T} \ar[d]_{\mu \Box \id}& &\und{T}^{\Box 4} \ar[urr]^{\id \Box \mu \Box \id}  \ar[rr]^{\id \Box s_{\Delta H} \Box \id}&& \und{T}^{\Box 4} \ar[urr]^{\id \Box \mu \Box \id} \ar[rr]^{\mu \Box \id_2}&& \und{T}^{\Box 3} \ar[ur]^{\mu \Box \id} &\\
 \und{T}\Box_{\und{R}}  \und{T} \ar[rru]^{\xi \Box \id_2} \ar[rr]^{\eta \Box \id_2} \ar@/_20ex/[rrrrrrruu]_{\id_2}& & \und{T}^{\Box 3} \ar[u]^{\sigma \Box \id_2} \ar@/_3ex/[rrrru]_{\id_3}&&
 & && 
}\]
The pentagon on the left commutes because the maps don't interfere. The associativity of the multiplication ensures that the upper small triangle and the right-most parallelogram commute. The lower left triangle commutes  thanks to the definition of $\xi$.

In the diagram $s_{\Delta H}$ denotes the trace on $\und{T} \Box_{\und{R}} \und{T}$ that uses the diagonal $H$-action on both $\und{T}$-factors. As the $H$-action is multiplicative, we get that the parallelogram containing $s_H$ commutes. 

Due to \eqref{eq:identity1} we also get that the curved triangle commutes because only the terms in the trace that act via the unit element are left. The bottom part of the diagram commutes because $\eta$ is a unit for $\mu$. 

In total we get that the composite on the left which is  $\psi_R$ coincides with the composite $\theta$. As $\psi_L$ also coincides with $\theta$ we get that $\psi_R = \psi_L$. 

\bigskip
As $\psi_R$ uses the right unit isomorphism $\nu_R$ and then applies $\xi$ and a multiplication, it does not interfere when we multiply on the left and hence the diagram
\[ \xymatrix{\und{T} \Box_{\und{R}} \und{T} \ar[rr]^{\id \Box \psi_R} \ar[d]_{\mu}&& \und{T} \Box_{\und{R}} \und{T} \Box_{\und{R}} \und{T} \ar[d]^{\mu \Box \id} \\
 \und{T} \ar[rr]^{\psi_R}&& \und{T} \Box_{\und{R}} \und{T} }\]
commutes. Hence $\psi_R$ is a left $\und{T}$-module map. Dually, $\psi_L$ is a right $\und{T}$-module map. As $\psi_R = \psi_L$, this map is compatible with the $\und{T}$-bimodule structure.

Last but not least we show that $\psi_R = \sigma$ which then of course also implies $\psi_L = \sigma$. To that end consider the diagram
\[ \xymatrix{
\und{T} \ar[rrrr]^{i^{(e)}} \ar[d]_{\nu_R}& &&&  \prod_H \und{T} \ar[r]^{\chi^{-1}} & \und{T} \Box_{\und{R}} \und{T} \\ 
\und{T} \Box_{\und{R}}\und{R} \ar[rr]^{\id \Box \eta}&& \und{T} \Box_{\und{R}} \und{T}
\ar[rr]^{\id \Box i^{(e)}}&& \und{T} \Box_{\und{R}} \prod_H \und{T} \ar[u]_m \ar[r]^{\id \Box \chi^{-1}} & \und{T} \Box_{\und{R}} \und{T} \Box_{\und{R}} \und{T} \ar[u]_{\mu \Box \id}
  }\]
Here, $m$ is the left $\und{T}$-module structure on $\prod_H \und{T}$ that is given by the multiplication in $\und{T}$ in every factor. By direct inspection one sees that $m \circ (\id \Box \chi) = \chi \circ (\mu \otimes \id)$ and therefore the right square commutes and the left diagram commutes because $i^{(e)}$ only includes $\und{T}$ into the $e$-factor of the product and $\eta$ is the unit map of the multiplication. 

Therefore the top row describing $\sigma$ coincides with the other composite which is $\psi_R$. 
\end{proof}

We will now provide some examples of Galois extensions of Tambara functors.

\begin{prop} \label{prop:c2galois}
  If $K \ra L$ is a $C_2$-Galois extension of fields, then $\und{K}^c
  \rightarrow \und{L}^\fix$ is a $C_2$-Galois extension of $C_2$-Tambara functors where we define the $C_2$-Galois action as the Weyl-action on $\und{L}^\fix$.  
\end{prop}  

\begin{proof}
As $C_2$ is abelian, the $C_2$-Galois action commutes with the $C_2$-Weyl action of the $C_2$-Tambara functor.

We first consider the case where $K$ is a field of characteristic prime to $2$. 
We claim that the isomorphism of $C_2$-Tambara functors from Corollary
\ref{cor:boxfix} 
$\und{L}^\fix \Box_{\und{K}^c} \und{L}^\fix \cong \und{(L \otimes_K L)}^\fix$ can
be combined with the isomorphism of commutative rings $\chi \colon L \otimes_K L
\cong \prod_{C_2} L$ to induce an isomorphism of $C_2$-Tambara functors
\[ \und{L}^\fix \Box_{\und{K}^c} \und{L}^\fix \cong \prod_{C_2} \und{L}^\fix. \]
A priori the map $\chi$ is equivariant with respect to the $C_2$-action that
acts on the right hand tensor factor in $L \otimes_K L$ and acts on
the indexing set in $\prod_{C_2} L$. However, in $\und{(L \otimes_K
  L)}^\fix$ the $C_2$-Weyl 
action is diagonally on $L \otimes_K L$ and acts on the $L$ in $\prod_{C_2} L$.

We show that $\chi$ is also equivariant with respect to this action:
If $C_2 = \langle \tau | \tau^2 = e\rangle$, then $L = K(\alpha)$ with
$\alpha^2 = a \in K$ and $\tau(\alpha) = -\alpha$. We denote elements
in $\prod_{C_2} L$ by $(b,c)$ where $b$ is in the $1$-coordinate, and
$c$ is in the $\tau$-coordinate. The generator $1 \otimes 1$ is fixed
under $\tau$. The generators $1 \otimes \alpha$, $\alpha \otimes 1$
and $\alpha \otimes \alpha$ sit in commutative diagrams 
\begin{equation} \label{eq:charnot2}
  \xymatrix{1 \otimes \alpha \ar@{|->}[r]^\chi \ar@{|->}[d]_\tau&
    (\alpha,-\alpha) \ar@{|->}[d]^\tau\\ -1 \otimes \alpha
    \ar@{|->}[r]^\chi & (-\alpha, \alpha)} \xymatrix{\alpha \otimes 1
    \ar@{|->}[r]^\chi \ar@{|->}[d]_\tau& (\alpha,\alpha)
    \ar@{|->}[d]^\tau\\ -\alpha \otimes 1 \ar@{|->}[r]^\chi & (-\alpha,
    -\alpha)} \xymatrix{\alpha \otimes \alpha \ar@{|->}[r]^\chi
    \ar@{|->}[d]_\tau& (\alpha^2,-\alpha^2) \ar@{=}[r] & (a,-a)
    \ar@{|->}[d]^\tau\\ (-\alpha) \otimes (-\alpha) \ar@{|->}[r]^\chi &
    (\alpha^2, -\alpha^2) \ar@{=}[r] & (a,-a).}\end{equation} 

Therefore $\chi \colon L \otimes_K L \cong \prod_{C_2} L$ as commutative $C_2$-$K$-algebras with the above action and this shows
\[ \und{L}^\fix \Box_{\und{K}^c} \und{L}^\fix \cong \und{(L \otimes_K L)}^\fix \cong \und{\Big(\prod_{C_2} L\Big)}^\fix \cong \prod_{C_2} \und{L}^\fix.\]
The last isomorphism is due to the fact that $(-)^\fix$ is a right adjoint functor. 

The fixed point property $\Big(\und{L}^\fix\Big)^{C_2} \cong \und{K}^c$ is easy
to see, because $L^{C_2} =K$ and $K^{C_2} = K$. 

\bigskip
Assume now that the characteristic of $K$ is $2$. Then $L$ is an Artin-Schreier
extension and is of the form $L = K(\alpha)$ with $\alpha^2 + \alpha + a = 0$ for some $a \in K$ and $\tau(\alpha) = \alpha+1$. By \cite[Proposition 3.4]{lrz}
we know that generators in $\und{L}^\fix \Box_{\und{K}^c} \und{L}^\fix$ are of the form $1 \otimes 1 \in K \otimes_K K \cong K$ and $[\lambda \alpha \otimes \alpha]$ with $\lambda \in K$. We claim that also in this case we get an isomorphism
of $C_2$-Tambara functors
\begin{equation} \label{eq:boxchar2} \und{L}^\fix \Box_{\und{K}^c} \und{L}^\fix \cong \und{\Big(L \otimes_K L\Big)}^\fix.\end{equation}
Again, at the free level we take the identity map whereas at level $C_2/C_2$ we send $1 \otimes 1$ to $1 \otimes 1$ but
\[ [\lambda \alpha \otimes \alpha] \mapsto \lambda \alpha \otimes 1 + \lambda \otimes \alpha + \lambda \otimes 1. \]
We need to show that this is an isomorphism of commutative rings which together with the identity at the free level constitutes an isomorphism of $C_2$-Tambara functors. The $\lambda$-factor only contributes notational complexity. The product
\[ [\alpha \otimes \alpha]^2 = [\alpha \otimes \alpha \cdot \res\tr(\alpha \otimes \alpha)]\]
can be calculated to give $[\alpha \otimes \alpha] +a[\alpha \otimes 1] + a[1\otimes \alpha]$ but as $[\alpha \otimes 1]\sim 1 \otimes 1 \sim [1 \otimes \alpha]$ we are left with $[\alpha \otimes \alpha]$. In the target
\[ (\alpha \otimes 1 + 1 \otimes \alpha + 1 \otimes 1)^2 = \alpha \otimes 1 + 1 \otimes \alpha + 1 \otimes 1\]
thus the map respects the multiplication. A direct calculation shows that the map is also compatible with restriction, transfer and norm. 

As in the first case we have to show that the map $\chi \colon L
\otimes_K L \cong \prod_{C_2} L$ is also equivariant with respect to
the diagonal action in the source and the action on $L$ in the
target. Again, this is checked on generators. We present the most
complicated case and chase $\alpha \otimes \alpha$: 

\begin{equation} \label{eq:char2}
  \xymatrix{ \alpha \otimes \alpha \ar@{|->}[d]_\tau \ar@{|->}[r]^\chi & (\alpha^2, \alpha(\alpha+1)) \ar@{=}[r] & (\alpha+a,a) \ar@{|->}[d]^\tau \\
(\alpha +1) \otimes (\alpha + 1) \ar@{|->}[r]^\chi & ((\alpha+1)^2, (\alpha+1)\alpha) \ar@{=}[r] &  (\alpha^2+1,a)=(\alpha+a+1,a)
  }\end{equation}
The rest of the argument agrees with the one in the first case. 
\end{proof}

\begin{rem} \label{rem:char2}
  \begin{enumerate}
  \item[]
    \item
  Note that we proved in \eqref{eq:boxchar2} that
  \[ \und{L}^\fix \Box_{\und{K}^c} \und{L}^\fix \cong \und{\Big(L \otimes_K L\Big)}^\fix\] if $K \subset L$ is a $C_2$-Galois extension, even if $K$ has characteristic $2$, so in particular $|C_2|=2$ is \emph{not} invertible. 
\item 
If $K \subset L$ is an $H$-Galois extension of fields, then in general $\und{K}^c \ra \und{L}^\fix$ won't be an $H$-Galois extension of $H$-Tambara functors: If $H$ is not abelian, then the Galois action does not have to commute with the Weyl action despite the fact that they agree. Even if $H$ is abelian, the diagrams analogous to \eqref{eq:charnot2} and \eqref{eq:char2} won't necessarily commute.  
\end{enumerate}
\end{rem}  

\begin{ex}
By Proposition \ref{prop:c2galois} the $C_2$-Galois extension $\mathbb{F}_2
\subset \mathbb{F}_4$ yields a $C_2$-Galois extension of Tambara
functors  $\und{\mathbb{F}}_2^c \ra
\und{\mathbb{F}}_4^\fix$ and by Theorem \ref{thm:galoisseparable} this
is a separable extension. We 
proved in \cite[Lemma 3.8]{lrz} that this extension is flat. Therefore
Proposition \ref{prop:flip-separable} and Corollary \ref{cor:flipfix-separable} imply  that  
\[ \und{\pi}_0\cL_{S^\sigma}^{C_2,\und{\F}_2^c}(\und{\F}_4^\fix) \cong \und{\F}_4^\fix \]
is the only non-trivial homotopy group of the simplicial Tambara
functor $\cL_{S^\sigma}^{C_2,\und{\F}_2^c}(\und{\F}_4^\fix)$. We invite
the reader to try to check this by hand. For me this was the
motivating example to develop the notions of separability and Galois
extensions.  
\end{ex}

The above examples of $C_2$-Galois extension of $C_2$-Tambara functors also shows that the map $s_H = \sum_{h \in H} h \colon \und{T} \ra \und{R}$ does not have to be surjective: 
\begin{prop}
  There are $H$-Galois extensions of $G$-Tambara functors $\und{R} \ra \und{T}$
  such that the map $s_H \colon \und{T} \ra \und{R}$ is not surjective at every level. 
\end{prop}
\begin{proof}
  Consider a $C_2$-Galois extension of fields $K \ra L$ of
  characteristic $2$ and the corresponding $C_2$-Galois extension of
  $C_2$-Tambara functors $\und{K}^c \ra \und{L}^\fix$.  At level
  $C_2/C_2$ the map $s_H$ is
  \[ (e + \tau) \colon \und{L}^\fix (C_2/C_2) = L^{C_2} = K \ra \und{K}^c(C_2/C_2) = K\]
  and hence it sends any $x \in K$ to $2x=0$. 
\end{proof}

\begin{rem}
  In the Galois theory of commutative rings the surjectivity of the
  trace $s_H \colon T \ra R$ is an important means to ensure the
  projectivity and faithful flatness of $T$ as an $R$-module. If
  $\und{R} \ra \und{T}$ is an $H$-Galois extension of $G$-Tambara
  functors, then $\und{R}(G/e) \ra \und{T}(G/e)$ is an $H$-Galois
  extension of commutative rings and therefore we know that $s_H
  \colon \und{T}(G/e) \ra \und{R}(G/e)$ is surjective, $\und{T}(G/e)$
  is projective and faithfully flat as an $\und{R}(G/e)$-module. 

  For $\und{K}^c \ra \und{L}^\fix$ coming from a $C_2$-Galois
  extension $K \subset L$ we showed in \cite[Lemma 3.8]{lrz} by
  different methods that $\und{L}^\fix$ is projective over
  $\und{K}^c$.  We don't have a general argument for the projectivity
  of  Galois extensions of Tambara functors.  

\end{rem}

In analogy with the corresponding notion for modules over ring spectra
\cite[Definition 4.3.1]{rognes} we define faithfulness for modules over a
Tambara functor as follows:

\begin{defn}
Let $\und{R}$ be a $G$-Tambara functor and let $\und{M}$ be an
$\und{R}$-module. Then $\und{M}$ is \emph{faithful}, if for all
$\und{R}$-modules $\und{N}$ the triviality of $\und{N} \Box_\und{R}
\und{M}$ implies the triviality of $\und{N}$.
\end{defn}  
Note that this clashes with the classical notion of faithfulness in
the category of modules over a commutative ring where one requires the
annihilator of the module to be trivial.

If the group order $|H|$ is invertible in $\und{R}$, then we do get faithfulness with the help of the usual proof (see \cite[Lemma 6.4.3]{rognes}):

\begin{prop}
Assume that $\und{R} \ra \und{T}$ is an $H$-Galois extension of
$G$-Tambara functors and assume that $|H|$ is invertible in $\und{R}$,
then $\und{T}$ is a faithful $\und{R}$-module. 
\end{prop}  

\begin{proof}
Consider the composite $\xymatrix@1{\und{R} \ar[r]^\eta & \und{T}
  \ar[r]^{s_H} & \und{R}}$ which sends an $x \in \und{R}(G/K)$ to
$|H|x$. As $|H|$ is invertible, this shows that $\und{R}$ splits off
$\und{T}$ via maps of $\und{R}$-modules. Assume that for an $\und{R}$-module $\und{M}$ we get
$\und{M} \Box_{\und{R}} \und{T} \cong 0$. Then $\und{M} \cong
\und{M}\Box_{\und{R}} \und{R}$ splits of $0$, and hence $\und{M} \cong
0$ and $\und{T}$ is faithful as an $\und{R}$-module. 
\end{proof}

\begin{ex}
If $R \ra T$ is an $H$-Galois extension of commutative rings, then we saw in Proposition \ref{prop:constantgalois} that  
$\und{R}^c \ra \und{T}^c$ is an $H$-Galois extension in $G$-Tambara functors such that $\und{T}^c$ is flat over $\und{R}^c$.  As the trace map $s_H \colon T \ra R$ is a surjective map, we get surjectivity of the trace map at every level. This insures faithfulness because for every $\und{R}^c$-module $\und{M}$ the map 
\[ \xymatrix@1{\und{M} \Box_{\und{R}^c} \und{T}^c \ar[rr]^{\id \Box s_H} &&\und{M} \Box_{\und{R}^c} \und{R}^c \cong \und{M}} \]
is surjective. 
\end{ex}

\begin{thm} Assume that $G$ and $H$ are finite groups and that $|G|$ is prime to the characteristic of a field $k$. 
  If $j \colon k \ra L$ is a $G \times H$-Galois extension of fields and $K = L^H$, then the $H$-Galois extension 
  $\phi \colon K \ra L$  induces an $H$-Galois extension of $G$-Tambara functors
  $\varphi \colon \und{K}^\fix \ra \und{L}^\fix$ with $\varphi_e = \phi$.   
\end{thm}

\begin{proof}
  As the $G$ and $H$ action on $L$ commute, the $H$ action on $L$ gives a
  well-defined action on the $G$-Tambara functor $\und{L}^\fix$. At every level
  $G/U$ we get
  \[ \und{L}(G/U)^H = (L^U)^H \cong (L^H)^U = K^U = \und{K}^\fix(G/U).\]
Thus we only have to check the unramified condition, thus we consider
\[ \chi \colon \und{L}^\fix \Box_{\und{K}^\fix} \und{L}^\fix \ra \prod_H \und{L}^\fix. \]
As $|G|$ is invertible in $k$ (and hence in $K$ and $L$) we know that 
$\und{L}^\fix \Box\und{L}^\fix \cong \und{(L \otimes L)}^\fix$ and 
$\und{L}^\fix \Box \und{K}^\fix \Box \und{L}^\fix \cong \und{(L \otimes K \otimes L)}^\fix$. We showed in the proof of \cite[Proposition 4.1]{lr} that taking the 
fixed point Tambara functor commutes with coequalizers and therefore
\[ \und{L}^\fix \Box_{\und{K}^\fix} \und{L}^\fix \cong \und{(L \otimes_K L)}^\fix.\] 
But as $K \subset L$ is an $H$-Galois extension the map of commutative rings
$\chi \colon L \otimes_K L \ra \prod_H L$, $\chi(a \otimes b) = (a\cdot h(b))_{h \in H}$ is an isomorphism. Thus we obtain
\[ \und{(L \otimes_K L)}^\fix \cong \und{\Big(\prod_H L\Big)}^\fix. \]
As the functor that takes an abelian group with a  $G$-action to its fixed
point $G$-Mackey functor is right adjoint to the functor that evaluates a
$G$-Mackey functor at the free level, it commutes with products, and therefore
we obtain that $h$ is an isomorphism of $G$-Tambara functors 
\[ \und{L}^\fix \Box_{\und{K}^\fix} \und{L}^\fix \cong \prod_H \und{L}^\fix. \]
\end{proof}

\begin{ex}
  The $C_4$-Galois extension of fields
  $\Q(\sqrt{2}) \hookrightarrow \Q(\sqrt{2}, \zeta_5)$ induces a $C_4$-Galois
  extension of $C_2$-Tambara functors $\und{\Q(\sqrt{2})}^\fix \ra \und{\Q(\sqrt{2}, \zeta_5)}^\fix$. 
  \end{ex}

In non-equivariant algebra the ring of integers $\Z$ is the inital
object in the category of commutative rings and is separably closed, \ie,
it has no non-trivial idempotents and every non-trivial finite Galois
extension of $\Z$ has non-trivial idempotents. In the category of
commutative ring spectra the sphere spectrum is separably closed 
\cite[Theorem 10.3.3]{rognes}. We obtain an analogous 
statement in equivariant algebra. Note that the Burnside Tambara
functor cannot have non-trivial idempotents because it is an initial object. 

\begin{thm} \label{thm:asepclosed}
  The Burnside $G$-Tambara functor $\und{A}$ is separably closed for all
  finite groups $G$, \ie, if $\und{A} \ra \und{T}$ turns $\und{T}$
  into a non-trivial $H$-Galois extension of $\und{A}$ as a $G$-Tambara functor,
  then $\und{T}$ has a non-trivial idempotent, so that $\und{T} \cong
  \und{T}_1 \times \und{T}_2$ as $G$-Tambara functors with
  non-trivial $\und{T}_1$ and $\und{T}_2$.  
\end{thm}

In the following proof we need to distinguish between Burnside $G$-Tambara
functors for several groups $G$ and for the duration of the proof we 
use $\und{A}_G$ as notation for the initial object in the category
of $G$-Tambara functors. 

\begin{proof}
Assume that $\und{A}_G \ra \und{T}$ is an $H$-Galois
extension with $|H|>1$. Then at the free level, we get a non-trivial $H$-Galois extension
$\und{A}_G(G/e) = \Z \ra \und{T}(G/e) =:T$, and as $\Z$ is separably
closed  the commutative ring $T$ splits as $T \cong \prod_H \Z$ (see
the proof of \cite[Proposition 10.3.2]{rognes}). If we choose an 
isomorphism $\phi \colon \prod_H \Z$, then $\phi$ does not have to be
$G$-equivariant, but we can transfer the $G$-action along $\phi$ as
usual so that for all $(x_h)_{h \in H} \in \prod_H \Z$ the action is defined as
\[ g.((x_h)_{h\in H}) := \phi^{-1}(g\phi((x_h)_{h\in H})). \]
This is again a $G$-action by ring maps and $\phi$ is equivariant with
respect to this action.

As $\Z$ is an integral domain with trivial ring automorphisms, the
ring automorphisms of $\prod_H \Z$ are induced by permutations of the
indexing set $H$.

We consider three cases:

1) If the $G$-action on $\prod_H \Z$ is trivial, then the idempotents
$\varepsilon_h \in \prod_H \Z$ that have a $1$ in spot $h \in H$ and zeros
everywhere else give a complete set of $G$-invariant orthogonal idempotents and
by \cite[Proposition 3.6]{wisdom-clarified} we obtain a
splitting of $\und{T}$ as a $G$-Tambara functor.

2) Assume now that the orbit of $\varepsilon_e$ is $\{\varepsilon_{h_1},
\ldots \varepsilon_{h_\ell}\}$ for some $1 < \ell < |H|$. Then the
element
\[ \rho := \sum_{i=1}^\ell \varepsilon_{h_i} \]
is a non-trivial idempotent and is $G$-invariant by construction. Similarly, the
element $1-\rho$ is a non-trivial $G$-invariant idempotent and hence again by
\cite[Proposition 3.8]{wisdom-clarified} we obtain a splitting of
$\und{T}$ as a $G$-Tambara functor.

3) If the $G$-orbit of $\varepsilon_e$ is
the set $\{\varepsilon_h, h \in H\}$, then we can write the $G$-ring
$T$ as $\coind_K^G\Z$ for some subgroup $K < G$ where $\coind_K^G\Z$
denotes the coinduction of rings as in \cite[Definition
2.23]{wisdom-clarified}. But then $\prod_H \Z \cong \prod_{G/K} \Z$ and
hence $|H| = [G:K]$. 

Brun showed in \cite[\S 2]{brun} that there is a left adjoint to the functor that takes a
$G$-Tambara functor $\und{R}$ and sends it to the commutative $G$-ring
$\und{R}(G/e)$. As $\Z$ is the initial object in commutative
$G$-rings, the left adjoint sends $\Z$ to the Burnside $G$-Tambara
functor $\und{A}_G$. Therefore, a morphism of commutative $G$-rings from $\Z$ to
$\coind_K^G\Z$ is adjoint to a morphism of $G$-Tambara functors
$\und{A}_G \ra \coind_K^G\und{A}_K$ because  $(\coind_K^G\und{A}_K)(G/e)
= \coind_K^G\Z$. Therefore $\und{T} \cong \coind_K^G\und{A}_K$.

Then we get at level
$G/G$ that $\coind_K^G\und{A}_K(G/G) = \und{A}_K(K/K)$ with trivial
$H$-action. As we assume that  $\und{A}_G \ra  \coind_K^G\und{A}_K$ is
an $H$-Galois extension, $\und{A}_G(G/G)$ had to be isomorphic to
$\und{A}_K(K/K)$, which is not true unless $G = K$. But this implies
that $|H|=1$.  

\end{proof}

\section{No naive Tambara analogue of the Gaussian integers} \label{sec:nogauss}
The Gaussian integers $\Z[i]$ provide an important example of a
ramified extension. The conjugation action has $\Z$ as fixed points
but the extension $\Z \ra \Z[i]$ is \emph{not} a $C_2$-Galois
extension of commutative rings because the map $\xi \colon \Z[i]
\otimes \Z[i] \ra \prod_{C_2} \Z[i]$ is not surjective. But the
extension $\Z \ra \Z[i]$ is only ramified at $2$ and
\[ \Z[\frac{1}{2}] \ra \Z[\frac{1}{2}][i]\]
\emph{is} a $C_2$-Galois extension of commutative rings. This works in
broader generality: for a Galois extension of number fields, inverting
all the ramified primes yields a Galois extension
of the corresponding rings of integers \cite[Theorem 4.1]{greither}. 

If one wants to mimic the construction of the Gaussian integers $\Z[i]
= \Z[x]/x^2+1$ in the setting of $C_2$-Tambara functors, then one
would rewrite them as a pushout 
\[ \Z[x]/x^2+1 = \Z[x] \otimes_{\Z[y]} \Z \]
where $y$ is mappd to $x^2 \in \Z[x]$ and to $-1$ in $\Z$. 

Blumberg and Hill provide very helpful explicit formulas for free
$C_2$-Tambara functors in \cite[\S 3]{bh-right}. Adjoining a free
generator at the trivial level $C_2/C_2$, $x_{C_2/C_2}$, to the Burnside Tambara
functor for $C_2$ yields 
\[ \und{A}[x_{C_2/C_2}]: 
  \xymatrix{\Z[t]/t^2-2t[x_*,n_*]/tn_* = tx_*^2 \ar@/_3ex/[d]_\res\\
    \Z[x] \ar@<0.4ex>@/_3ex/[u]_{\tr} \ar@<-0.4cm>@/_3ex/[u]_{\norm}}
\] 
with $\res(x_*) = x$, $\res(n_*) = x^2$, $\tr(p(x)) = tp(x_*)$ for any
$p(x) \in \Z[x]$ and the norm is determined by the norm in the
Burnside Tambara functor, by $\norm(x) = n_*$ and by an explicit
inductive formular for polynomials of higher degree \cite[Proposition 3.5]{bh-right}.

A -- probably too naive -- attempt to build a Tambara version of the
Gaussian integers is to form a
relative box product
\[ \und{G} = \und{A}[x_{C_2/C_2}] \Box_{\und{A}[y_{C_2/C_2}]}
  \und{A}\]
where at level $C_2/e$ we send $y$ to $x^2$ in
$\und{A}[x_{C_2/C_2}](C_2/e) = \Z[x]$
and to $-1 \in \und{A}(C_2/e) = \Z$. We then obtain $\Z[x]/x^2+1 =
\Z[i]$ at the free level.
At both levels we have (quotients of) polynomial algebras, so we can
talk about the degrees of homogeneous elements. 

As in the non-equivariant case we assume the following about the maps
in the pushout diagram
\[ \xymatrix@1{ \und{A}[x_{C_2/C_2}] & \und{A}[y_{C_2/C_2}]\ar[r]^(0.6)\psi
    \ar[l]_\varphi &\und{A}:       } \]
\begin{enumerate}
\item
  The $C_2$-action is $\und{A}$-linear and is trivial on
  $\und{A}[y_{C_2/C_2}]$.
\item
The map  $\varphi$ is a morphism of Tambara functors that doubles the
degree.
\item
  The map $\psi$ is a morphism of Tambara functors. 
\end{enumerate}

However, we show the following no-go result:

\begin{prop}
There is no way to define a
$C_2$-action on $\und{G}$ which satisfies the above conditions (1),
(2) and (3), such that $\und{A} \cong \und{G}^{C_2}$. 

In particular, inverting $2$ does \emph{not} yield a $C_2$-Galois
extension $\und{A}[\frac{1}{2}] \ra \und{G}[\frac{1}{2}]$. 
\end{prop}
\begin{proof}
As we define the $C_2$-action on the free level as
\[ \tau(x) = -x\]
this implies
\[ \tau(n_*) = \tau(\norm(x)) = \norm(-x) = \norm(-1)n_* = (t-1)n_*.  \]
Note that this already implies that elements of the form $\beta tn_*$
are fixed by $\tau$.

\medskip 
We claim that we have no choice but to define $\tau(x_*)$ as
$(1-t)x_*$: As we assume that the $C_2$-action is degree-preserving,
we know that
\[ \tau(x_*) = (a+bt)x_* \text{ with } a,b \in \Z.\]
As $\tau$ has to act via morphisms of Tambara functors we get 
\[ \tau(tx_*) = \tau(\tr(x)) = \tr(-x) = -tx_* \]
and this results in the condition 
\begin{equation} \label{eq:x*} a + 2b = -1\end{equation}
The equality
\[ -x = \tau(x) = \tau(\res(x_*)) = \res\tau(x_*) = \res((a+bt)x_*) =
  (a +2b)x\]
yields the same condition.

Using $\tau^2(x_*) = x_*$ further implies that
\[a^2 + (2ab+2b^2)t = 1\] 
and by comparing this with \eqref{eq:x*} we get $b=-1$ and $a=1$,
hence
\begin{equation} \label{eq:taux}
\tau(x_*) = (1-t)x_*. 
\end{equation}  
Therefore the $C_2$-action is completely determined by condition (1). 

It is straightforward to check that elements of the form $(-2d+dt)x_*^{2n+1}$
are fixed points for all $d \in \Z$ and all $n \geq 0$, because $(-2d+dt)(1-t) =
-2d+2dt+dt-2dt = -2d+dt$, so we have fixed points in all odd degrees.

As we assume that $\varphi$ doubles the degree, these fixed points are
not in the image of $\varphi_{C_2}$ and hence they give non-trivial
fixed points in $\und{G}^{C_2}$, proving that $\und{A}$ is properly contained in $\und{G}^{C_2}$. 




\end{proof}

\begin{rem}
One can define Galois extensions of Green functors similar to the
notion for Tambara functors. In this setting, it \emph{is} possible to
realize an analogue of the Gaussian integers with $2$ inverted. By
\cite[Corollary 2.11]{bh-right} the free $C_2$-commutative Green
functor on a generator with 
trivial $C_2$-action $x_*$ is $\Z[x_*] \otimes \und{A}$ where all the
structure maps are determined by those of the $C_2$-Burnside Green functor
$\und{A}$. Note that the box product is the coproduct in the
category of Green functors \cite[Theorem 5.3]{blumberg-hill}. We can form the
pushout  
\[ \left(\Z\left[\frac{1}{2}, y_*\right]\otimes \und{A}\right)
  \Box_{\left(\Z\left[\frac{1}{2}, x_*\right] \otimes \und{A}\right)}
  \left(\Z\left[\frac{1}{2}\right] \otimes \und{A}\right)\]
by sending $x_*$ to $y_*^2$, respectively $-1$ and this yields
$\Z\left[\frac{1}{2},i\right] \otimes \und{A}$ because the functor $(-) \otimes
\und{A}$ is strong symmetric monoidal \cite[Proposition 
2.10]{bh-right}. If we define a $C_2$-Galois action on  
this Green functor by sending $i$ to $-i$, then
\[ \left(\Z\left[\frac{1}{2},i\right]\otimes \und{A}\right)^{C_2} \cong
  \Z\left[\frac{1}{2}\right] \otimes \und{A} \cong \und{A}\left[\frac{1}{2}\right]\]
and the map 
\[ \chi \colon  \left(\Z\left[\frac{1}{2},i\right] \otimes \und{A}\right)
  \Box_{\und{A}\left[\frac{1}{2}\right]} \left(\Z\left[\frac{1}{2},i\right] \otimes
    \und{A}\right) \ra
  \prod_{C_2} \left(\Z\left[\frac{1}{2},i\right] \otimes \und{A}\right)\]
is an isomorphism because $(-) \otimes \und{A}$ is strong symmetric
monoidal and $\Z\left[\frac{1}{2}\right] \ra
\Z\left[\frac{1}{2},i\right]$ is a $C_2$-Galois extension of
commutative rings. This procedure also enables us to import other
examples of Galois extensions of commutative rings into the setting of
Green functors by applying $(-)\otimes \und{A}$. 
\end{rem}

\section{Nullstellensatzian Tambara functors and Galois extensions} \label{sec:nullstellensatzian}
For Tambara funtors there is a notion of being field-like due to
Nakaoka \cite{nakaoka}. Field-like $G$-Tambara functors $\und{T} \neq 0$ can be
characterized  as those $G$-Tambara functors $\und{T}$ for which $\und{T}(G/e)$
has no non-trivial $G$-invariant ideals and that have injective
restriction maps associated to equivariant maps of finite transitive
$G$-sets and \cite[Theorem 4.32]{nakaoka}.  Typical examples arise
from  $G$-Galois extensions $K \subset L$ of fields. In this case the
fixed point $G$-Tambara functor $\und{L}^\fix$ is field-like. Wisdom
shows \cite[Corollary B]{wisdom-clarified} that every field-like
$G$-Tambara functor $\und{k}$ is of the form $\coind_H^G(\und{\ell})$
for some field-like $H$-Tambara functor $\und{\ell}$ such that
$\und{\ell}(H/e)$ is a field. Here, $\coind_H^G$ denotes the co-induction functor $\coind_H^G \colon H\tamb \ra G\tamb$. This functor is right adjoint to the restriction functor $i_H^G\colon G\tamb \ra H\tamb$.

Schuchardt, Spitz, and Wisdom investigate which Tambara functors
behave like algebraically closed objects. They do this using the concept of
\emph{Nullstellensatzian objects}, introduced in \cite{bsy} and show
that Nullstellensatzian $G$-Tambara functors $\und{k}$ are isomorphic
to the fixed point Tambara functor $\und{\und{k}(G/e)}^\fix$ and can
be expressed as the co-induction of an algebraically closed field
$\mathbb{F}$, $\und{k} \cong \coind_e^G(\mathbb{F})$ \cite[\S
5]{ssw}. So it is a natural question to ask how co-induction from the
trivial group to some finite group $G$ interacts
with Galois extensions:  
\begin{prop}
    If $R \ra T$ is an $H$-Galois extension of commutative rings for
    some finite group $H$, then
  $\coind_e^G(R) \ra \coind_e^G(T)$ is an $H$-Galois extension of
  $G$-Tambara functors for any finite group $G$.
\end{prop}
\begin{proof}
As $\coind_e^G$ is a functor from the category of commutative rings to
the category of $G$-Tambara functors, the $H$-action on $T$ by
$R$-algebra maps is transformed into an $H$-action on $\coind_e^G(T)$
by $\coind_e^G(R)$-algebra maps.

As $\coind_e^G(T)(G/K) = T(\bigsqcup_{[G:K]} e/e) = \prod_{[G:K]} T$,
the componentwise $H$-action yields 
\[ \coind_e^G(T)(G/K)^H \cong \coind_e^G(R)(G/K). \]

The co-induction functor is lax symmetric  monoidal and the left 
$\coind_e^G(R)$-module structure of $\coind_e^G(T)$ is induced by the
composite
\[ \coind_e^G(R) \Box \coind_e^G(T) \ra \coind_e^G(R \otimes T) \ra \coind_e^G(T)\] 
where the first map comes from the lax symmetric monoidality of
$\coind_e^G$, noting that the tensor product is the box product of
$e$-Tambara functors. The second map is the map induced by the
$R$-module structure of $T$. The right module structure can be
described similarly.

We claim that $\coind_e^G(T) \Box_{\coind_e^G(R)} \coind_e^G(T) \cong \coind_e^G(T \otimes_R T)$. The relative box product is the coequalizer of 
\[ \xymatrix{\coind_e^G(T) \Box \coind_e^G(R) \Box \coind_e^G(T)
    \ar@<0.5ex>[r]^(0.6)\ell \ar@<-0.5ex>[r]_(0.6)r &
    \coind_e^G(T) \Box \coind_e^G(T).  } \]
Again, we use Strickland's description of elements in box products, so at level $G/U$ for some $U < G$ of 
$\coind_e^G(T) \Box \coind_e^G(T)$  we can represent elements as $T_p(t \otimes t')$ with $t,t' \in \coind_e^G(T)(G/V)$ for some $V < U$ and $p \colon G/V \ra G/U$. Starting at the free level we have
\[ \coind_e^G(T) \Box \coind_e^G(T)(G/e) = (\prod_G T) \otimes (\prod_G T)\]
and we consider the element $(t,t,\ldots,t) \otimes (t',0,\ldots,0) \in \prod_G T \otimes \prod_G T$. As the transfer in a co-induction Tambara functor just induces addition
and as restriction gives a diagonal map, we obtain with Frobenius reciprocity 
\[ t \otimes t' = t \otimes \tr(t',0,\ldots,0) = \tr(\res(t) \otimes (t',0,\ldots,0)) = \tr((t,t,\ldots,t) \otimes (t',0,\ldots,0)) \]
and hence we can represent every element in every level of the box product as
a transfer coming from the free level and at the free level, we just have $(\prod_G T) \otimes (\prod_G T)$.

A similar argument shows that every element at every level of the three-fold box product $\coind_e^G(T) \Box \coind_e^G(R) \Box \coind_e^G(T)$ can be written as the transfer of an element coming from the free level and there we get
\[ (\prod_G T) \otimes (\prod_G R) \otimes (\prod_G T).\]
The coequalizer at the free level is
\[ (\prod_G T) \otimes_{(\prod_G R)} \otimes (\prod_G T)\]
and coequalizing the $\prod_G R$-module structure kills mixed terms, so we are left with $\prod_G (T \otimes_R T)$.

As $\chi \colon T \otimes_R T \ra \prod_H T$ is an isomorphism of
commutative rings, $\prod_G \chi$ is an isomorphism as well. As the free level determines the box products for co-inductions from the trivial group to $G$, we get
\[\coind_e^G(T) \Box_{\coind_e^G(R)} \coind_e^G(T) \cong \coind_e^G(T \otimes_R T) \cong \coind_e^G(\prod_H R)\] and as co-induction is a right adjoint, it commutes with products. 
\end{proof}

Assume now that $\mathbb{F}$ is an algebraically closed field and that $K$ is a field such that $K \subset \mathbb{F}$ is a
(possibly non-finite) Galois extension. Then we can write $\mathbb{F}$ as a filtered colimit
\[ \mathbb{F} = \bigcup_{\stackrel{K \subset L \text{Galois}}{[L:K] < \infty}} L .\]

For all these  \emph{finite} extensions $K \subset L$ with Galois group
$H = \mathrm{Gal}(L/K)$ we obtain a corresponding $\mathrm{Gal}(L/K)$-Galois extension of 
$G$-Tambara functors $\coind_e^G(K) \ra \coind_e^G(L)$ for all finite
groups $G$. Filtered colimits are built levelwise in Tambara functors and the co-induction functor
$\coind_e^G$ commutes with filtered colimits. Therefore we can express $\coind_e^G\mathbb{F}$ as a filtered colimit  
\[ \coind_e^G\mathbb{F} = \colim\limits_{\stackrel{K \subset L \text{Galois}}{[L:K] < \infty}} \coind_e^G(L).\]
Using \cite[Theorem 5.7]{ssw} this expresses every Nullstellensatzian
$G$-Tambara functor $\und{k} = \coind_e^G(\mathbb{F})$ as a filtered
colimit of finite Galois extensions.

\begin{rem}
In the non-equivariant Galois theory of commutative rings, $H$-Galois
extensions of the form $R \ra \prod_H R$ are frowned upon: the target
ring has non-trivial idempotents and this non-connectedness creates
problems for instance if one wants to define what a separable closure
is. But Schuchardt, Spitz, and Wisdom show that a Nullstellensatzian object
in the category of commutative $H$-rings is precisely an $H$-ring of
the form $\text{Fun}(H,\mathbb{F}) = \prod_H \mathbb{F}$ where
$\mathbb{F}$ is algebraically closed \cite[Theorem
5.13]{ssw}. Adjunction gives a canonical map $\mathbb{F} \ra \prod_H
\mathbb{F}$ where we forget about the $H$-action in the target. The
map, however, is the diagonal map which yields the $H$-Galois
extension of commutative rings $\mathbb{F} \ra \prod_H \mathbb{F}$. 
\end{rem}

\end{document}